\documentclass{article}
\usepackage{amsmath,mathrsfs,amssymb,amsthm,amscd}
\usepackage{color,epsfig,latexsym,graphicx,eucal,layout,fancyhdr}
\usepackage{enumerate}
\usepackage[]{fontenc}
\usepackage[all]{xy}
\usepackage[bookmarks=false]{hyperref}
\usepackage{graphics}
\usepackage{a4wide}
\usepackage{psfrag}

\newtheorem{thm}{Theorem}[section]

\newtheorem{prop}[thm]{Proposition}
\newtheorem{cor}[thm]{Corollary}
\newtheorem{defn}[thm]{Definition}
\newtheorem{example}[thm]{Example}

\newtheorem{rmk}[thm]{Remark}

\newcounter{EQNR}
\addtocounter{section}{0} 
\numberwithin{equation}{section}

\def\Empty{}
\newcommand\oplabel[1]{
  \def\OpArg{#1} \ifx \OpArg\Empty {} \else
    \label{#1}
  \fi}

\long\def\realfig#1#2#3#4{
\begin{figure}[tp]
\centerline{\psfig{figure=#2,width=#4}} \caption[#1]{#3}
\oplabel{#1}
\end{figure}}

\input{psbox}

\psfordvips

\begin{document}

\title{Spectral asymptotics on sequences of elliptically degenerating Riemann surfaces}
\author{26 Feb 2016}
\date{D. Garbin and J. Jorgenson
\footnote{The first author acknowledges support from a PSC-CUNY
grant. The second author acknowledges support from grants from the
NSF and PSC-CUNY. }}

\maketitle

\begin{abstract}\noindent
This is the second in a series of two articles where we study various aspects of the spectral theory associated to 
families of hyperbolic Riemann surfaces obtained through elliptic degeneration.  In the first article, we investigate the asymptotics
of the trace of the heat kernel both near zero and infinity and we show the convergence of small eigenvalues and corresponding eigenfunctions. Having obtained necessary bounds for the trace, this second article presents the behavior of several spectral invariants. Some of these invariants, such as the Selberg zeta function and the spectral counting functions associated to small eigenvalues below 1/4, converge to their respective counterparts on the limiting surface. Other spectral invariants, such as the spectral zeta function and the logarithm of the determinant of the Laplacian diverge. In these latter cases, we identify diverging terms and remove their contributions, thus regularizing convergence of these spectral invariants. Our study is motivated by a result from \cite{He 83}, which D. Hejhal attributes to A. Selberg, proving spectral accumulation for the family of Hecke triangle groups. In this article, we obtain a quantitative result to Selberg's remark. 
\end{abstract}

\tableofcontents

\newpage

\section{Introduction}
\label{Introduction}

\hskip 0.2in In the last section of the monumental second volume of \emph{Selberg trace formula for $\mathrm{PSL}(2,\mathbb{R})$}, D. Hejhal proves a statement that he attributes to A. Selberg, concerning the behavior of the zeros and poles of the scattering determinant for the Eisenstein series associated to the Hecke triangle groups $G_N$ as $N$ goes to infinity. Namely, for the Hecke triangle groups $G_N$ which is a subgroup of $\mathrm{PSL}(2,\mathbb{R})$ generated by the fractional linear transformations $z \mapsto -1/z$ and $z \mapsto z+ 2\cos(\pi/N)$ for $3 \le N \le \infty$, the parabolic Eisenstein series associated to the cusp at infinity has the following Fourier expansion 
\begin{align*}
E_N(z;s) = y^s + \phi_N(s)y^{1-s} + O(e^{-2\pi y}) \,\,,
\end{align*}
where the function $\phi_N(s)$ is referred to as the determinant of the scattering matrix (a 1-1 matrix in this case.)
The behavior for the zeros and the poles of $\phi_N(s)$ are the last two results in Hejhal's second volume on the trace formula, with zeros accumulating to the right of the critical line and the poles to the left of it. The precise statements of Theorem 7.11 and Corollary 7.12  in \cite{He 83} are as follows:\\
\hskip 0.2in \mbox{}
\emph{
Given $t_0 \in \mathbb{R}$ and $0 < \delta < 1$, the rectangle $[\frac{1}{2},\frac{1}{2}+\delta] \times [t_0 -\delta,t_0+\delta]$ must contain zeros of 
$\phi_N(s)$ and the rectangle $[\frac{1}{2} -\delta,\frac{1}{2}] \times [t_0 -\delta,t_0+\delta]$ must contain poles of 
$\phi_N(s)$ when $N$ is sufficiently large.\\
}
The latter result appears in the ending remarks of Selberg's G\"{o}ttingen lectures part 2. Hejhal also promises to explore more this topic in a third volume on the trace formula, volume that unfortunately has not been published. 

\hskip 0.2in The Hecke triangle groups is one instance of a family of hyperbolic Riemann surfaces which is elliptically degenerated. In the setting of the Hecke groups $G_N$, Hejhal shows that the Eisenstein series and the scattering determinants converge through degeneration. Motivated by this, one can ask questions about the behavior of other spectral invariants such as the Selberg zeta and the spectral zeta functions, and spectral counting functions, and eigenfunctions associated to families obtained through elliptic degeneration.    

\hskip 0.2in 
This paper is the second in a series of two articles where we study various spectral invariants through elliptic degeneration. In the first paper \cite{GJ 16}, we show the behavior through degeneration of the trace of the heat kernel.  After we identify the divergent contributions to the trace, we regularize the trace and produce bounds in both compact and non-compact but finite volume settings. Using these bounds, we are then able to show that the small eigenvalues converge. In this second paper, we investigate the behavior of other spectral invariants as mentioned above. Among the highlights here is the quantification of the rate of accumulation of the poles of the scattering determinant for the Hecke triangle groups. 

\hskip 0.2in 
We should note that the phenomenon of degeneration of hyperbolic Riemann surfaces has been study elsewhere. A series of articles are concerned with the asymptotic behavior on families compact and non-compact finite volume through hyperbolic degeneration obtained by pinching geodesics. 
Articles such as \cite{JoLu 95}, \cite{Wo 87}, \cite{JoLu 97a}, \cite{HJL 97} study heat kernels, Selberg zeta function, determinants of the Laplacian, and small eigenvalues while for instance \cite{DJ 98} looks at the hyperbolic degeneration of 3-manifolds. While the methods of proofs in the hyperbolic degeneration case are adapted here in the elliptic setting, we feel that it would not be mathematically honest to simply state so without carefully analyzing the details. At the same time, the methods here are used to give a quantitative statement to Selberg's remark, which is not present in the previously mentioned articles. 

\hskip 0.2in The paper is organized as follows. In Section \ref{Geometry of elliptic degeneration}  we describe the setting of elliptic degeneration. In Section \ref{Asymptotics of heat kernels and traces} we recall the main results from the first paper \cite{GJ 16}. In Sections \ref{Asymptotics of spectral measures} and \ref{Convergence of spectral counting functions and small eigenvalues} we present the behavior of spectral measures in general and spectral counting functions in particular, the latter of the two sections containing the result about accumulation of the poles of the scattering determinant for the Hecke triangle groups. In Section \ref{Spectral functions} we present the behavior of the spectral and Hurwitz zeta functions while in Section \ref{Selberg zeta and determinant of the Laplacian} we study the Selberg zeta and the determinant of the Laplacian. Section \ref{Integral kernels}  concludes the paper by remarks concerning the behavior for other integral kernels. 

\section{Geometry of elliptic degeneration}
\label{Geometry of elliptic degeneration}

\hskip 0.2in In this section we will present the notion of elliptic degeneration of hyperbolic Riemann surfaces of finite volume (compact or non-compact), having elliptic fixed points. Elliptic degeneration occurs when the orders of such elliptic fixed points are increasing without a bound. As these orders are running off to infinity, their corresponding cones turn into cusps.

\hskip 0.2in Let $M$ be a connected hyperbolic Riemann surface of finite volume, either compact or non-compact. For simplicity, let us assume that $M$ is connected, so then $M$ can be realized as the quotient manifold $\Gamma \backslash \mathbb H$, where $\mathbb H$ is the hyperbolic upper half space and $\Gamma$ is a discrete subgroup of $\textrm{SL}(2,\mathbb R) \slash \{ \pm1 \}$.
A non-identity element $\gamma \in \Gamma$ is called hyperbolic, parabolic, or elliptic,  if $\gamma$ is conjugated in $\textrm{SL}(2,\mathbb R)$ to a dilation, horizontal translation,  or rotation respectively. This is analogous to $|\textrm{Tr}(\gamma)|$ being greater than, equal, or less than 2, respectively.  Furthermore, an element $\gamma$ is called primitive, if it is not a power other than $\pm1$ of any other element of the group. With this in mind, a primitive hyperbolic, parabolic, or elliptic element $\gamma$ is conjugated to $\left(\begin{array}{cc}
e^{\ell_{\gamma}/2} & 0 \\ 0&e^{-\ell_{\gamma}/2} \end{array}\right)$, $\left(\begin{array}{cc} 1 & w_{\gamma} \\ 0& 1 \end{array}\right)$, or $\left(\begin{array}{cc}
\cos(\pi/q_{\gamma}) & \sin(\pi/q_{\gamma}) \\ -\sin(\pi/q_{\gamma})&\cos(\pi/q_{\gamma}) \end{array}\right)$ respectively.
Here $\ell_{\gamma}$ is the length of the simple closed geodesic on the surface $M$ in the homotopy class of $\gamma$,
$w_{\gamma}$ denotes the width of the cusp fixed by $\gamma$, and $2\pi/q_{\gamma}$ is the angle of the conical point fixed by $\gamma$. The positive integer $q_{\gamma}$ is the order of the centralizer subgroup of the elliptic element $\gamma$. We will say that the corresponding elliptic fixed point has order $q_{\gamma}.$

\hskip .20in For a given positive integer $q$, let $C_q$ denote
the infinite hyperbolic cone of angle $2\pi/q$. One can realize $C_q$
as a half-infinite cylinder
\begin{equation}\label{hyperbolicpolarcoordinates}
C_{q} = \{ (\rho, \theta): \rho > 0, \theta \in [0,2\pi)\}\,\,.
\end{equation}
equipped with the Riemannian metric
\begin{equation}
ds^2 = d\rho^2 + q^{-2} \sinh^2 (\rho) d\theta^2\,\,,
\end{equation}
having volume form
\begin{equation}
d\mu = q^{-1} \sinh (\rho) d\rho d\theta\,\,.
\end{equation}

\hskip 0.2in A fundamental domain for $C_{q}$ in the hyperbolic unit disc model is
provided by a sector with vertex at the origin and with angle
$2\pi/q$. In coordinates, we write $\{ \alpha \exp(i\phi): 0 \le
\alpha < 1, 0 \le \phi < 2\pi/q\}$. The hyperbolic metric on $C_{q}$
is the metric induced onto the fundamental domain viewed as a subset of
the unit disc endowed with its complete hyperbolic metric.  The
isotropy group which corresponds to this fundamental domain
consists of the numbers $\exp(2\pi i k/q) $ for $k = 1,2, \cdots ,q$
acting by multiplication.
Let $C_{q,\varepsilon}$ denote the submanifold of
$C_{q}$ obtained by restricting the first coordinate of $(\rho,
\theta)$ to $0 \le  \rho < \cosh^{-1}(1+\varepsilon q/2\pi)$. A
fundamental domain for $C_{q,\varepsilon}$ in the unit disc model is
obtained by adding the restriction that
$\alpha < (\varepsilon q/(4\pi+\varepsilon q))^{1/2}.$ An elementary
calculation shows that the volume of this manifold
$\textrm{vol}(C_{q,\varepsilon}) = \varepsilon$, and the length of
the boundary of $C_{q,\varepsilon}$ is $(4\pi \varepsilon/q +
\varepsilon^2)^{1/2}$. For $\varepsilon_1 < \varepsilon_2$ one can
show that the length between the boundaries of the two nested cones
$C_{q,\varepsilon_1}$ and $C_{q,\varepsilon_2}$ is
\begin{equation*}
d_{\mathbb{H}}(\partial C_{q,\varepsilon_1}, \partial
C_{q,\varepsilon_2}) = \log \Bigg(\frac{\varepsilon_2 q + 2\pi +
\sqrt{\varepsilon_2 q(4\pi+\varepsilon_2q)}} {\varepsilon_1 q + 2\pi
+ \sqrt{\varepsilon_1 q(4\pi+\varepsilon_1q)}}\Bigg).
\end{equation*}

\hskip .20in Let $C_{\infty}$ denote an infinite cusp. A
fundamental domain for $C_{\infty}$ in the upper half plane is given
by the set $\{x+iy : y > 0, 0 < x < 1\}$. A fundamental domain for $C_{\infty}$
in the upper half plane is obtained
by identifying the boundary points $iy$
with $1+iy$. The isotropy group that corresponds to the above
fundamental domain consists of $\mathbb{Z}$ acting by addition.
As before, let $C_{\infty,\varepsilon}$ denote the submanifold of
$C_{\infty}$ obtained by restricting the $y$ coordinate of the
fundamental domain given above to $y > 2\varepsilon$. Easy
computations show that $\textrm{vol} (C_{\infty,\varepsilon}) =
\varepsilon/2$, and the length of the boundary of
$C_{\infty,\varepsilon}$ is also $\varepsilon/2$.

\hskip .20in Let $q = (q_1,q_2, \cdots, q_m)$ be a
vector of the orders of elliptic fixed points. In this case we
define $C_{q} = \cup_{k=1}^{m} C_{q_k}$. We similarly define $C_{q,
\varepsilon}$ as a union over the components of
$q$. We say that the vector $q$ approaches infinity if and
only if each of its components approach infinity.
Consequently, the Riemannian manifold $C_{q}$ (with $m$ connected
components) converges to $m$ copies of the limit Riemannian manifold
$C_{\infty}$ as $q \to \infty$. Similarly, $C_{q,
\varepsilon}$ converges to $m$ copies of $C_{\infty,\varepsilon}$.
We shall write these limits as $m \times C_{\infty}$ and $m \times
C_{\infty,\varepsilon}$.

\realfig{ed4}{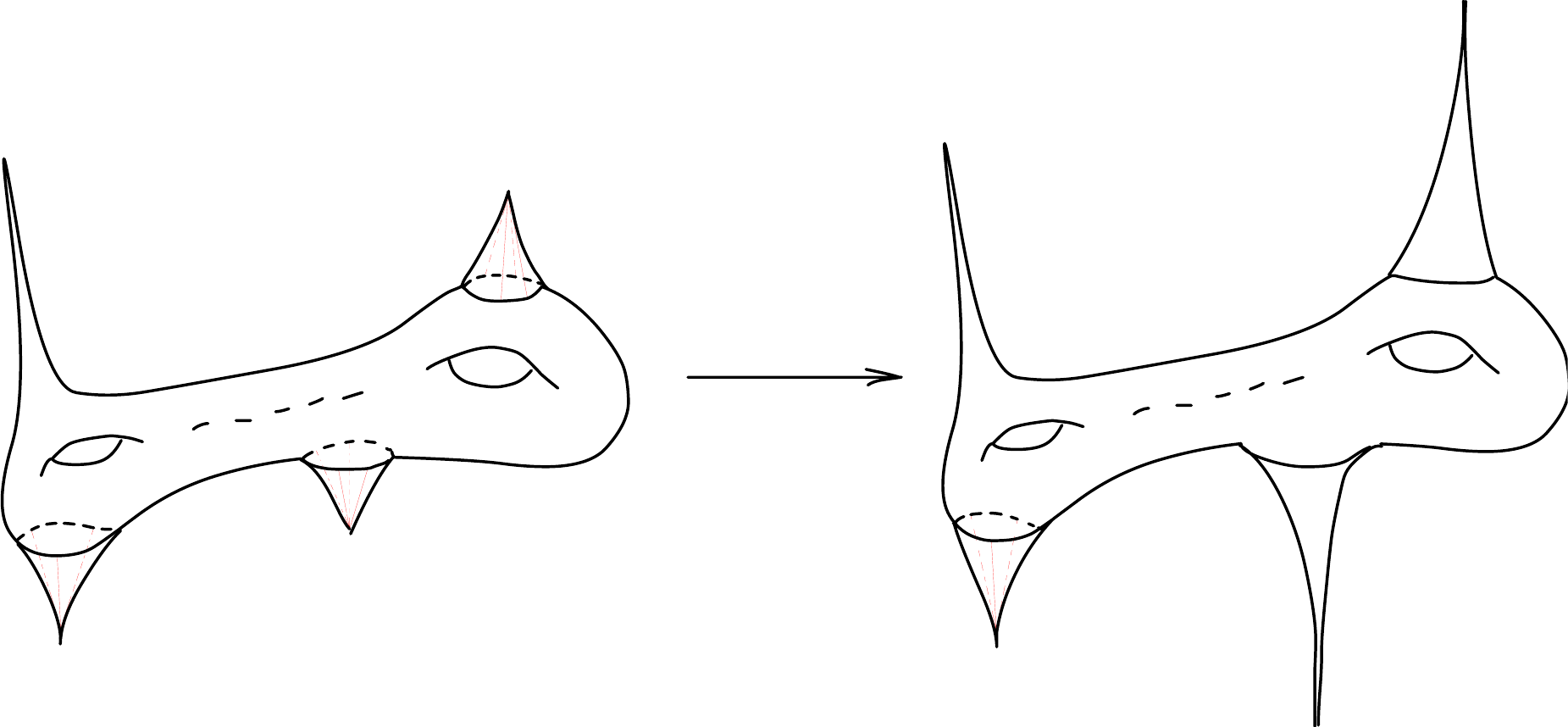}{Elliptic degeneration of
$q_1$ and $q_2$ $$ {
 \at{-5.83\pscm}{1.5\pscm}{$\frac{2\pi}{q_1}$}
 \at{-5.1\pscm}{4.2\pscm}{$\frac{2\pi}{q_2}$}
 \at{-7.73\pscm}{0.8\pscm}{$\frac{2\pi}{q_3}$}
 \at{-4.15\pscm}{3\pscm}{$q \to \infty$}
 \at{-8.8\pscm}{3\pscm}{$M_{q}$}
 \at{1\pscm}{3\pscm}{$M_{\infty}$}
}
$$}{9cm}

\hskip .20in
A family of finite volume hyperbolic surfaces
$M_{q}$ parametrized by the $m$-vector $q$ will be called
an elliptically degenerating surface if it has the following
properties (see Fig. \ref{ed4}):
\begin{enumerate}[a)]
\item For any $\varepsilon < 1/2$, the surface $C_{q, \varepsilon}$ (with $m$ components) embeds
isometrically into $M_{q}$.
\item As $q \to \infty$, $M_{q}$ converges to a complete, hyperbolic surface
$M_{\infty}$ in the following sense. $M_{\infty}$ contains $m$
embedded copies of $C_{\infty,\varepsilon}$ which is the limit of
$C_{q, \varepsilon} \subset M_{q}$. The geometry of
$M_{q} \backslash C_{q, \varepsilon}$ converges to the
geometry of $M_{\infty} \backslash (m \times
C_{\infty,\varepsilon})$.
\end{enumerate}

In the above, $m \times
C_{\infty,\varepsilon}$ refers to the ``new'' cusps of $M_{\infty}$,
that is, the cusps which developed from degeneration. In particular,
for every $q$, it is possible to identify points $x(q)$ and $y(q)$ on
$M_{q} \backslash C_{q, \varepsilon}$ such that
$\lim_{q \to \infty} d_{q}(x(q), y(q)) =
d_{\infty}(x(\infty), y(\infty))$. Henceforth, we shall suppress the
$q$ dependence of points which are identified during
degeneration and simply write $x$ and $y$. The volume forms induced
by the converging metrics also converge uniformly on $M_{q}
\backslash C_{q, \varepsilon}$, and all such measures are
absolutely continuous with respect to each other. In general, the
hyperbolic volume form occurring in an integral will be denoted by
$d\mu$ with an appropriate subscript when needed (for example,
$d\mu_{q}$). Length measure will be denoted by $d\rho$.

\hskip .20in The description of the degeneration of $M_{q}$ to
the limit surface $M_{\infty}$ also applies to the degeneration of
$C_{q}$ and $C_{q, \delta}$ (with $\varepsilon < \delta$)
to their limit surfaces, $m \times C_{\infty}$ and $m \times
C_{\infty, \delta}$ respectively.

\section{Asymptotics of heat kernels and traces}
\label{Asymptotics of heat kernels and traces}
\hskip 0.2in We begin this section with a brief account of the spectral theory and the heat kernel associated to a hyperbolic Riemann surface of finite volume. After defining the regularized, hyperbolic, elliptic, and degenerating traces of the heat kernel, we present a particular instance of the Selberg trace formula. We end this section, by harvesting some of the results from \cite{GJ 16} concerning the behavior thorough elliptic degeneration of the regularized trace. For brevity, we only present these results without proof, referring the reader to \cite{GJ 16} where elaborate accounts are given.

\hskip 0.2in
In this paper we consider hyperbolic surfaces having conical singularities (see \cite{Ju 95}), surfaces realized as the action of discrete groups $\Gamma$ of $\mathrm{PSL}(2,\mathbb{R})$ acting on $\mathbb{H}$. The conical singularities are present once the group $\Gamma$ contains elements (other than the identity) having fixed points. Such is the case with the full modular group $\mathrm{PSL}(2,\mathbb{Z})$.
In particular, let $M$ be a compact hyperbolic surface, having  $n$ marked points $\{c_i\}_{i=1}^{n}.$ The Riemannian metric $g$ on $M$ is called conically singular hyperbolic metric if and only if for every $i = 1,..,n$ there exists a chart $(U_i,\mu_i)$ about the point $c_i$ isometrically mapping $U_i$ to a hyperbolic cone model with associated angle $\alpha_i$. The metric $g$ induces a compatible complex structure on $M\backslash \{c_i\}_{i=1}^n$ and together with the charts  $\{(U_i,\mu_i)\}_{i=1}^{n}$, provide a complex structure on $M$. Given such structure, there exist a unique complete hyperbolic metric on $M\backslash \{c_i\}_{i=1}^n$ for which each $c_i$ is a cusp. 

\hskip 0.2in
As the surfaces in consideration have conical points, there is a way to extend the domain on which the Laplace operator acts so that it is self-adjoint. The Friedrichs procedure is one such possible extension and we use it here for our spectral purposes. Namely, the domain of the extension is the closure in $L^2(M)$ of the space
\begin{align*}
D = \left\{ f \in L^2(M) \,\,: \,\, \int_M (<\textrm{grad} f, \textrm{grad} f > + f^2)d\mu < \infty\ \textrm{ and } \int_{\partial \textrm{cusp}} f d\mu= 0 \right\}
\end{align*} 
where $d\mu$ denotes the hyperbolic volume form and the domain of integration for the second integral is a horocycle.
For details of the above construction we refer the reader to \cite{LP 76}, \cite{CdV 83}, \cite{Ji 94}, and \cite{Ju 95}. Throughout this paper, we will refer to the pseudo-Laplacian above as simply the Laplace operator.

\hskip 0.2in
Let $\Delta_M$ denote the non-negative Laplace operator on the surface $M$. Consider the heat operator $\Delta_M + \partial_t$ acting on functions $u : M\times {\mathbb R}^{+} \mapsto \mathbb R$ which are $C^2(M)$ and $C^1({\mathbb R}^{+})$. Then the heat kernel associated to $M$ is the minimal integral kernel which inverts the heat operator. Namely, the heat kernel is a function $K_M: \mathbb{R} \times M \times M \mapsto \mathbb R$ satisfying the following conditions. For any function $f\in C^2(M)$ consider the integral transform
\begin{align*}
u(t,x) = \int_M K_M(t,x,y)f(y)d\mu_M(y)\,\,.
\end{align*}
Then the following differential and initial time conditions are met
$$
\Delta_x u+ \partial_t u = 0
\,\,\,\,\,\text{\rm and} \,\,\,\,\,
f(x) = \lim_{t\to 0^+} u(t,x)\,\,.
$$

\hskip 0.2in If $M$ is compact, then the spectrum of the  Laplace operator is discrete, consisting of eigenvalues $0=\lambda_0 < \lambda_1 \le \lambda_2 \le \,\,\, \to \infty$ counted with multiplicity. Associated to these eigenvalues there is complete system $\{\phi_n(x)\}_{n=0}^{\infty}$
of orthonormal eigenfunction of the Laplace operator on $M.$  For $t>0$ and $x,y \in M$, the heat kernel has the following realization \begin{align}\label{heatkernelexpansioncompact}
K_M(t,x,y) = \sum_{n=0}^{\infty} e^{-\lambda_n t} \phi_n(x) \phi_n(y)\,\,,
\end{align}
and the sum converges uniformly on $[t_0,\infty)\times M \times M$ for fixed $t_0>0$ (see for instance \cite{Ch 84}.)

\hskip 0.2in If $M$ is not compact, the spectrum has a discrete part as well as a continuous part in the real interval $[1/4,\infty]$. The continuos spectrum comes from the parabolic Eisenstein series $E_{\text{par};M,P}(s,z)$ associated to the each cusp $P$ of $M$. In such case, the spectral expansion has the following form (coming from \cite{He 83})
\begin{align}\label{heatkernelexpansionnoncompact}
\notag K_M(t,x,y) =& \sum_{\text{discrete}} e^{-\lambda_n t} \phi_n(x) \phi_n(y)\\
&+ \frac{1}{2\pi} \sum_{\text{cusps }P} \int_{0}^{\infty} e^{-(1/4+r^2)t} E_{\text{par};M,P}(1/2+ir,x) \overline{E_{\text{par};M,P}(1/2+ir,y)} dr\,\,.
\end{align}

\hskip .20in  Let $K_{\mathbb{H}}(t, \tilde{x}, \tilde{y})$
denote the heat kernel on the upper half plane. Recall that
$K_{\mathbb{H}}(t, \tilde{x}, \tilde{y})$ is a function of $t$ and
the hyperbolic distance $d = d_{\mathbb{H}}(\tilde{x}, \tilde{y})$
between $\tilde{x}$ and $\tilde{y}$, so
\begin{equation*}
K_{\mathbb{H}}(t, \tilde{x}, \tilde{y}) = K_{\mathbb{H}}(t, d).
\end{equation*}
Quoting from page 246 of \cite{Ch 84}, we have for $d>0$
\begin{equation}\label{heatkernelonH1}
K_{\mathbb{H}}(t, d) = \frac{\sqrt{2}e^{-t/4}}{(4\pi t)^{3/2}}
\int_{d}^{\infty} \frac{u e^{-u^2/4t}du}{\sqrt{\cosh u - \cosh d}}
\end{equation}
while for $d=0$
\begin{equation}\label{heatkernelonH0}
K_{\mathbb{H}}(t, 0) = \frac{1}{2\pi } \int_{0}^{\infty}
e^{-(1/4+r^2)t} \tanh(\pi r)rdr.
\end{equation}

\begin{rmk}\label{complexrealtime}
\emph{It is possible to extend the heat kernel
to complex valued time. For time $z \in \mathbb{C}$, write $z=t+is$ with
$t>0$. Then we have
\begin{equation*}
K_{\mathbb{H}}(z, d) = \frac{\sqrt{2}e^{-z/4}}{(4\pi z)^{3/2}}
\int_{d}^{\infty} \frac{u e^{-u^2/4z}du}{\sqrt{\cosh u - \cosh d}}\,\,,
\end{equation*}
and setting $\tau=|z|^2/t$,  yields the bound
\begin{align*}
|K_{\mathbb{H}}(z,d)| &\le
\frac{\sqrt{2}e^{-t/4}}{(4\pi)^{3/2}(t^2+s^2)^{3/4}}
\int_{d}^{\infty} \frac{ue^{-u^2/4\tau}du}{\sqrt{\cosh u - \cosh
d}}\\
&\le e^{s^2/4t}t^{-3/2}(t^2+s^2)^{3/4} K_{\mathbb{H}}(\tau,d).
\end{align*}
}
\end{rmk}

\hskip 0.2in For any hyperbolic Riemann surface $M \simeq \Gamma \backslash \mathbb H$, one can express
the heat kernel as a periodization of the heat kernel of the
hyperbolic plane. Let $x$ and $y$ denote points on $M$ with lifts
$\tilde{x}$ and $\tilde{y}$ to $\mathbb{H}$. Then we can write the
heat kernel on $M$ as
\begin{equation}\label{periodization}
K_M(t,x,y) = \sum_{\gamma \in \Gamma} K_{\mathbb{H}}(t,
d_{\mathbb{H}}(\tilde{x}, \gamma \tilde{y})).
\end{equation}

Denote by $H(\Gamma), P(\Gamma)$, and $E(\Gamma)$ complete sets of
$\Gamma$-inconjugate primitive hyperbolic, parabolic, and elliptic elements, respectively, of the group  $\Gamma$. If $M$ is compact, then $P(\Gamma)$ is
empty. Let $\Gamma_{\gamma}$ denote the centralizer of $\gamma \in
\Gamma$. If $\gamma$ is a hyperbolic or a parabolic element then
$\Gamma_{\gamma}$ is isomorphic to the infinite cyclic group. If
$\gamma$ is elliptic then its centralizer is isomorphic
to the finite cyclic group of order $q_{\gamma}$. In each instance, the centralizer is generated by a primitive element. We can use
elementary theory of Fuchsian groups (see for instance \cite{McK 72}) to write the periodization (\ref{periodization}) as
\begin{align*}
K_M(t,x,y) = K_{\mathbb{H}}(t,\tilde{x},\tilde{y}) &+
 \sum_{\gamma \in P(\Gamma)} \sum_{n=1}^{\infty}\sum_{\kappa \in
\Gamma_{\gamma} \backslash \Gamma}
K_{\mathbb{H}}(t,\tilde{x},\kappa^{-1} \gamma^n \kappa
\tilde{y})\\
&+  \sum_{\gamma \in H(\Gamma)} \sum_{n=1}^{\infty}\sum_{\kappa \in
\Gamma_{\gamma} \backslash \Gamma}
K_{\mathbb{H}}(t,\tilde{x},\kappa^{-1} \gamma^n \kappa
\tilde{y})\\
&+  \sum_{\gamma \in E(\Gamma)} \sum_{n=1}^{q_{\gamma}-1}
\sum_{\kappa \in \Gamma_{\gamma} \backslash \Gamma}
K_{\mathbb{H}}(t,\tilde{x},\kappa^{-1} \gamma^n \kappa \tilde{y}).
\end{align*}
Using the above decomposition we define the parabolic contribution
(i.e. the contribution coming from the parabolic elements) to the
trace of the heat kernel by
\begin{align*}
PK_M(t,x) = \sum_{\gamma \in P(\Gamma)} \sum_{n=1}^{\infty}
\sum_{\kappa \in \Gamma_{\gamma} \backslash \Gamma}
K_{\mathbb{H}}(t,\tilde{x},\kappa^{-1} \gamma^n \kappa \tilde{x})
\end{align*}
and in a similar manner we define the hyperbolic contribution and
elliptic contribution which we denote by $HK_M(t,x)$ and $EK_M(t,x)$
respectively.

\begin{defn}\label{standardheattrace}
For a connected hyperbolic surface $M$, we define the regularized or standard heat trace $\mathrm{Str}K_M(t)$ by
\begin{align*}
\text{\rm STr}K_M(t) = \text{\rm HTr}K_M(t) + \text{\rm ETr}K_M(t) +
\mathrm{vol}(M)K_{\mathbb{H}}(t,0)\,\,,
\end{align*}
where the hyperbolic and elliptic traces $\text{\rm HTr}K_M(t)$ and $\text{\rm ETr}K_M(t)$ are given by
\begin{align*}
\text{\rm HTr}K_M(t) = \int_M HK_M(t,x) d\mu(x) \,\,\,\,\, \text{and} \,\,\,\,\,
\text{\rm ETr}K_M(t) = \int_M EK_M(t,x) d\mu(x)\,\,,
\end{align*}
respectively.
If $M$ is a hyperbolic Riemann surface of finite volume, but not
connected, each trace can be defined as the sum of the traces associated to each connected component of $M$.
\end{defn}

\hskip 0.2in The following result due to Selberg \cite{Se 56}
evaluates the integral representation that defines the hyperbolic trace, namely
\begin{equation}\label{hyperbolictrace}
\text{\rm HTr}K_M(t) = \frac{e^{-t/4}}{\sqrt{16\pi t}}
\sum_{\gamma \in H(\Gamma)}\sum_{n=1}^{\infty}
\frac{\ell_{\gamma}}{\sinh(n
\ell_{\gamma}/2)}e^{-(n\ell_{\gamma})^2/4t},
\end{equation}
while the elliptic trace can be expressed as
\begin{equation}\label{elliptictrace}
\text{\rm ETr}K_M(t) = \frac{e^{-t/4}}{\sqrt{16\pi t}}\sum_{\gamma\in
E(\Gamma)}\sum_{n=1}^{q_{\gamma}-1} \frac{1}{q_{\gamma}}
\int_{0}^{\infty}
\frac{e^{-u^2/4t}\cosh(u/2)}{\sinh^2(u/2)+\sin^2(n\pi/q_{\gamma})}du.
\end{equation}
For detailed accounts of (\ref{hyperbolictrace}) and (\ref{elliptictrace}), we refer the reader to Theorem 1.3 of \cite{JoLu 97b} and Theorem 2.5 of \cite{GJ 16} respectively.

\hskip 0.2in An alternative expression for the elliptic heat trace, namely
\begin{align}\label{Hejhalelliptictrace}
\text{\rm ETr}K_M(t) = \sum_{\gamma\in E(\Gamma)}
\sum_{n=1}^{q_{\gamma}-1}\frac{e^{-t/4}}{2q_{\gamma}\sin(n\pi/q_{\gamma})}
            \int_{-\infty}^{\infty} \frac{e^{-2\pi nr / q_{\gamma} - tr^2}}{1+e^{-2\pi r}}dr
\end{align}
can be found in \cite{He 76} on page 351
or \cite{Ku 73} on pages 100-102. One can use the Parseval formula to show that the two seemingly different expressions (\ref{elliptictrace}) and (\ref{Hejhalelliptictrace}) for $\textrm{Etr}K_M(t)$ are equal (see Remark 2.6 in \cite{GJ 16}).

\begin{rmk}\label{Selbergtraceformula}
\emph{In the case $M$ is compact, the standard trace $\text{\rm STr}K_M(t)$ is simply
the trace of the heat kernel. One immediately obtains from (\ref{heatkernelexpansioncompact}) the spectral aspect of the standard trace,
\begin{align}\label{regularizedtracecompactspectral}
\text{\rm STr}K_M(t) = \int_{M} K_M(t,x,x)d\mu(x) = \sum_{n=0}^{\infty} e^{-\lambda_nt}\,\,.
\end{align}
On the other hand, Definition \ref{standardheattrace} and the various aforementioned integral representations ((\ref{heatkernelonH0}), (\ref{hyperbolictrace}), (\ref{Hejhalelliptictrace})), give the geometric side of the standard trace, namely
\begin{align}\label{regularizedtracegeometric}
\notag \text{\rm STr}K_M(t) =&
\frac{\textrm{vol}(M)}{4\pi}
\int_{-\infty}^{\infty}e^{-(r^2+1/4)t}\tanh(\pi r)rdr\\
\notag &+ \sum_{\gamma
\in H(\Gamma)} \sum_{n=1}^{\infty} \frac{\ell_{\gamma}}{\sinh(n
\ell_{\gamma}/2)} \frac{e^{-t/4}}{\sqrt{16\pi t}} e^{-(n\ell_{\gamma})^2/4t}\\
&+\sum_{\gamma\in E(\Gamma)} 
\sum_{n=1}^{q_{\gamma}-1}\frac{e^{-t/4}}{2q_{\gamma}\sin(n\pi/q_{\gamma})}
            \int_{-\infty}^{\infty} \frac{e^{-2\pi nr / q_{\gamma} - tr^2}}{1+e^{-2\pi r}}dr\,\, .
\end{align}
The combination of (\ref{regularizedtracecompactspectral}) and (\ref{regularizedtracegeometric}) represent an instance of the Selberg trace formula as applied to the function $f(r) = e^{-tr^2}$ and its Fourier transform  $\hat{f}(u) =
(4\pi t)^{-1/2}e^{-u^2/4t}$. \\
\text{\hskip 0.2in} One can use this special case to generalize the trace formula to a larger class of functions. First, denote by $r_n$ the solutions to $\lambda_n = 1/4+ r_n^2$. The non-negativity of the eigenvalues imply that for each $n$ there are two solutions $r_n$ which are either opposite real numbers or complex conjugate numbers in the segment $[-i/2,i/2]\,\,.$ \\
\text{\hskip 0.2in}To continue, let $h(t)$ be any measurable function for which $h(t)e^{(1/4+\varepsilon)t}$ is in $L^1(\mathbb R)\,$ for some $\varepsilon>0\,.$ Multiply the right-hand side of (\ref{regularizedtracecompactspectral}) and (\ref{regularizedtracegeometric}) by $h(t)e^{t/4}$ and integrate from 0 to $\infty$ with respect to $t.$ Set
\begin{align*}
H(r)=\int_{0}^{\infty} h(t) e^{-r^2t} dt\,\,.
\end{align*}
By rewriting the absolute integrand of $H(r)$ as $|h(t)e^{(1/4+\varepsilon)t})|\cdot |e^{-(r^2+1/4+\varepsilon)t})|$ and recalling the imposed conditions on $h(t)$, it easily follows that $H(r)$ is analytic inside the horizontal strip $|\text{Im}(r) | \le 1/2 + \varepsilon'$ for some $\varepsilon'>0$ depending on $\varepsilon\,\,.$ The Fourier transform of $H(r)$ has the form
\begin{align*}
\hat{H}(u)=\int_{0}^{\infty}  h(t) \frac{1}{\sqrt{4\pi t}}e^{-u^2/4t} dt\,\,.
\end{align*}
Putting these facts together yields the Selberg trace formula in the compact case, namely
\begin{align}\label{STFcompact}
\notag \sum_{r_n} H(r_n) =&
\frac{\textrm{vol}(M)}{4\pi}
\int_{-\infty}^{\infty}H(r) \tanh(\pi r)rdr\\
\notag &+ \sum_{\gamma
\in H(\Gamma)} \sum_{n=1}^{\infty} \frac{\ell_{\gamma}}{2\sinh(n
\ell_{\gamma}/2)}\hat{H}(n\ell_{\gamma})\\
&+\sum_{\gamma\in
E(\Gamma)}\sum_{n=1}^{q_{\gamma}-1} \frac{1}{2q_{\gamma}\sin(n\pi/q_{\gamma})}
            \int_{-\infty}^{\infty} H(r) \frac{e^{-2\pi nr / q_{\gamma}}}{1+e^{-2\pi r}}dr\,\, ,
\end{align}
where the sum on the left is taken over $r_n \in (0,\infty) \cup [0,i/2]\,\,.$ We note that (\ref{STFcompact}) above agrees with the formula in Theorem 5.1 of \cite{He 76}, with $\chi$ being the trivial character of the group $\Gamma$. 
\\
\\
\text{\hskip 0.2in} In the case $M$ is non-compact, the regularized trace equals the trace of the heat kernel minus the contribution of the parabolic conjugacy classes. While the geometric side of the regularized trace is precisely as in (\ref{regularizedtracegeometric}), the spectral side has the following presentation
\begin{align}\label{regularizedtracenoncompactspectral}
\notag \text{\rm STr}K_M(t) =& \sum_{C(M)} e^{-\lambda_n t}  - \frac{1}{4\pi} \int_{-\infty}^{\infty} e^{-(r^2+1/4)t} \frac{\phi'}{\phi}(1/2+ir)dr\\
\notag &+ \frac{p}{2\pi} \int_{-\infty}^{\infty} e^{-(r^2+1/4)t} \frac{\Gamma'}{\Gamma}(1+ir)dr\\
&- \frac{1}{4}\Big(p - \text{Tr }\Phi(1/2) \Big)e^{-t/4} + \frac{p \log 2}{\sqrt{4\pi t}} e^{-t/4}\,\,,
\end{align}
where   $C(M)$ denotes a set of eigenvalues associated to $L^2$ eigenfunctions on $M$, $\phi(s)$ the determinant of the scattering matrix $\Phi(s)$, $\Gamma(s)$ the Euler gamma function, while $p$ the number of cusps of $M$ (see page 313 of \cite{He 83}).\\
\text{\hskip 0.2in} One can use the same argument as in the compact case to obtain the formal Selberg trace formula in the non-compact case. While the geometric side doesn't change (see the right-hand side of (\ref{STFcompact}), the spectral side is as follows:
\begin{align}\label{STFnoncompact}
\notag \text{\rm spectral side} =& \sum_{r_n} H(r_n)  - \frac{1}{4\pi} \int_{-\infty}^{\infty} H(r) \frac{\phi'}{\phi}(1/2+ir)dr\\
\notag &+ \frac{p}{2\pi} \int_{-\infty}^{\infty} H(r)\frac{\Gamma'}{\Gamma}(1+ir)dr\\
&- \frac{1}{4}\Big(p - \text{Tr }\Phi(1/2) \Big)H(0) + p \log(2) \hat{H}(0)\,\,.
\end{align}
}
\end{rmk}

\begin{rmk}
\emph{
Returning to the special case of the trace formula given by (\ref{regularizedtracegeometric}), we note the following. For the first term in the right hand side of (\ref{regularizedtracegeometric}),
the identity contribution, we can write
\begin{align*}
\text{\rm ITr}K_M(t) =
\frac{\textrm{vol}(M)e^{-t/4}}{4t}
\int_{0}^{\infty} e^{-tr^2} \mathrm{sech}^2(\pi r)dr
\end{align*}
using integration by parts. Furthermore, for any $t\ge 0$, the integral can be bounded as follows
\begin{align*}
\int_{0}^{\infty} e^{-tr^2} \mathrm{sech}^2(\pi r)dr \le \int_{0}^{\infty} \mathrm{sech}^2(\pi r)dr = \frac{1}{\pi}
\end{align*}
with equality when $t = 0.$ It then follows that the identity trace has the following asymptotics
\begin{align}\label{identitytraceasymp}
\text{\rm ITr}K_M(t) = 
\left\{
\begin{array}{lr}
\cfrac{\textrm{vol}(M)}{4\pi t } + O(1), & \textrm{at } t = 0\\ \\
O(e^{-t/4 }), &  \textrm{at } t = \infty\,\,.
\end{array}
\right.
\end{align} 
The hyperbolic trace, the second term in the geometric side of the trace (\ref{regularizedtracegeometric}), has the following asymptotics
\begin{align}\label{hyperbolictraceasymp}
\text{\rm HTr}K_M(t) =
\left\{
\begin{array}{rl} 
O(e^{-c/t }),& \textrm{at } t = 0\\
O(e^{-t/4 }),&  \textrm{at } t = \infty\,\,.
\end{array}
\right.
\end{align}  
For a detailed account of these see Theorem 1.1 in \cite{JoLu 97b}. To continue, the integrals in the elliptic trace can be bounded as follows: For any primitive elliptic element $\gamma \in E(\Gamma)$ and $1\le n<q_{\gamma}$, we have 
\begin{align*}
\int_{-\infty}^{\infty} \frac{e^{-tr^2 - 2\pi nr / q_{\gamma}} }{1+e^{-2\pi r}}dr 
= & \int_{0}^{\infty} \frac{e^{-tr^2 - 2\pi nr / q_{\gamma}} }{1+e^{-2\pi r}}dr +  \int_{0}^{\infty} \frac{e^{-tr^2 +2\pi nr / q_{\gamma}} }{1+e^{2\pi r}}dr \\
& \le \int_{0}^{\infty} e^{-tr^2 - 2\pi (n / q_{\gamma})r}dr +  \int_{0}^{\infty} e^{-tr^2 - 2\pi(1-n/ q_{\gamma})r}dr\,\,. 
\end{align*}  
Now for $b>0$, the function $G_b(t)$ given by the Gaussian integral
\begin{align*}
G_b(t) = \int_{0}^{\infty} e^{-tr^2-br} dr
\end{align*} 
is defined for any $t\ge 0$. Furthermore, since the limits of $G_b(t)$ at $t=0$ and at $t=\infty$ are $b^{-1}$ and $0$ respectively, the integrals in the elliptic trace are finite for all $t\ge0$.
Consequently, the elliptic trace has the following behavior
\begin{align}\label{elliptictraceasymp}
\text{\rm ETr}K_M(t) = 
\left\{
\begin{array}{rl}
O(1),&\textrm{at } t = 0\\
O(e^{-t/4 }),&\textrm{at } t = \infty\,\,.
\end{array}
\right.
\end{align}
Putting all these together, the combination of  (\ref{regularizedtracecompactspectral}), (\ref{regularizedtracegeometric}), (\ref{identitytraceasymp}), (\ref{hyperbolictraceasymp}), and (\ref{elliptictraceasymp}) give the asymptotic behavior for the standard trace of the heat kernel in the compact setting. Namely, at $t=0$, we have
\begin{align}\label{standardtraceasymp0}
\text{\rm Str}K_M(t) = \sum_{n=0}^{\infty} e^{-\lambda_n t} = \frac{\textrm{vol}(M)}{4\pi t} + O(1)\,\,,
\end{align}
while at $t=\infty$
\begin{align}\label{standardtraceasympinfty}
\text{\rm Str}K_M(t) = 1+O(e^{-ct}) \,\,
\end{align}
for positive constant $c$.
Furthermore, if we denote by $N(\lambda) = \textrm{card} \{\lambda_n: \lambda_n \le \lambda\}$, then we can write the above expansion as follows
\begin{align*}
\int_{0}^{\infty} e^{-\lambda t}dN(\lambda) = \frac{\textrm{vol}(M)}{4\pi t} + O(1) \textrm{ at } t = 0\,\,.
\end{align*}
The Tauberian-Karamata theorem then gives an instance of Weyl's law as applied in the setting of hyperbolic Riemann surfaces
\begin{align}\label{Weyl'slaw}
N(\lambda) \sim \frac{\textrm{vol}(M)}{4\pi}\lambda  \textrm{ at } \lambda = \infty\,\,.
\end{align}
}
\end{rmk}

\hskip 0.2in The next result presents the behavior through degeneration of the heat kernel and its derivatives. Namely, we have the following theorem. For brevity, we only state the result. For details, we refer the reader to \cite{JoLu 95} and Theorem 1.3 of \cite{JoLu 97a} which one can easily adapt to the elliptic degeneration setting.

 \begin{thm}\label{convergenceheatkernels}
Let $R_{q}$ denote either $M_{q}$ or $C_{q}$. For $i
= 1, 2$, let $\nu_i = \nu_i(q)$ be a tangent vector of unit
length based at $x_i \in R_{q}$ which converges as $q \to
\infty$. Denote by $\partial_{\nu_i,x_i}$ the directional
derivative with respect to the variable $x_i$ in the direction
$\nu_i$. Assume that either $x_1$ or $x_2$ is not a degenerating
conical point. Then
\begin{align}\label{convergenceheatkernels1}
\lim_{q \to \infty} K_{R_{q}}(z,x_1,x_2) &=
K_{R_{\infty}}(z,x_1,x_2) \\
\label{convergenceheatkernels2}\lim_{q \to \infty}
\partial_{\nu_i,x_i} K_{R_{q}}(z,x_1,x_2) &=
\partial_{\nu_i,x_i} K_{R_{\infty}}(z,x_1,x_2) \textrm{ for }
i=1,2\\
\label{convergenceheatkernels3}\lim_{q \to \infty}
\partial_{\nu_1,x_1} \partial_{\nu_2,x_2} K_{R_{q}}(z,x_1,x_2)
&= \partial_{\nu_1,x_1}
\partial_{\nu_2,x_2} K_{R_{\infty}}(z,x_1,x_2)
\end{align}

\hskip .2in
\begin{enumerate}[a)]
\item Let A be a bounded set in the complex plane with $\inf_{z\in A}
\rm{Re}(z) > 0.$ For any $\varepsilon > 0$, the convergence is
uniform on $A \times R_{q} \backslash C_{q,\varepsilon}
\times R_{q} \backslash C_{q,\varepsilon}$.

\item We define $D_{\varepsilon,\varepsilon'}$ to be an $\varepsilon'$
neighborhood of the diagonal of $R_{q} \backslash
C_{q,\varepsilon} \times R_{q} \backslash
C_{q,\varepsilon}$. That is,
\begin{equation*}
D_{\varepsilon,\varepsilon'} = \{ (x_1,x_2) \in R_{q}
\backslash C_{q,\varepsilon} \times R_{q} \backslash
C_{q,\varepsilon} : d(x_1,x_2) < \varepsilon' \}
\end{equation*}
Let $B$ be a bounded set in the complex plane with $\inf_{z\in B}
\rm{Re}(z) \ge 0$. For any $\varepsilon > 0$ and $\varepsilon' > 0$,
the convergence is uniform on $B \times ((R_{q} \backslash
C_{q,\varepsilon} \times R_{q} \backslash
C_{q,\varepsilon}) \backslash D_{\varepsilon,\varepsilon'})$.
\end{enumerate}
\end{thm}

\hskip 0.2in To continue, let us define the degenerating trace of the heat kernel.
Denote by $DE(\Gamma)$ a subset of the elliptic conjugacy classes
$E(\Gamma)$, corresponding to the cones we wish to degenerate into cusps. It then follows that the degenerating heat trace $\mathrm{Dtr}K_M(t)$ can be expressed as
\begin{equation}\label{degeneratingtrace}
\text{\rm DTr}K_M(t) = \frac{e^{-t/4}}{\sqrt{16\pi t}}\sum_{\gamma\in
DE(\Gamma)}\sum_{n=1}^{q_{\gamma}-1} \frac{1}{q_{\gamma}}
\int_{0}^{\infty}
\frac{e^{-u^2/4t}\cosh(u/2)}{\sinh^2(u/2)+\sin^2(n\pi/q_{\gamma})}du.
\end{equation}

\hskip 0.2in One of the staple ingredient in this paper is the convergence through elliptic degeneration of the regularized trace minus the degenerating trace on $M_q$ to the regularized trace on the limiting surface $M_{\infty}.$ The result below presents a two-fold aspect of this behavior, whose details can be found in Theorem 5.4 of \cite{GJ 16}.

\begin{thm}\label{pointwiseanduniformconvergence}
Let $M_{q}$ denote an elliptically degenerating family of
compact or non-compact hyperbolic Riemann surfaces of finite volume
converging to the non-compact hyperbolic surface $M_{\infty}.$ 
\begin{enumerate}[(a)]
\item (Pointwise) For fixed $z=t+is$ with $t>0$, we have
\begin{align*}
\lim_{q \to
\infty}[\text{\rm HTr}K_{M_{q}}(z)+\text{\rm ETr}K_{M_{q}}(z)-\text{\rm DTr}K_{M_{q}}(z)]
=\text{\rm HTr}K_{M_{\infty}}(z)+\text{\rm ETr}K_{M_{\infty}}(z).
\end{align*}
\item (Uniformity) For any $t>0$, there exists a constant $C$ (depending on $t$) such
that for all $s \in \mathbb{R}$ and all $q$, we have the bound
\begin{align*}
|\text{\rm HTr}K_{M_{q}}(z)+\text{\rm ETr}K_{M_{q}}(z)-\text{\rm DTr}K_{M_{q}}(z)| \le
C(1+|s|)^{3/2}.
\end{align*}
\end{enumerate}
\end{thm}

Furthermore, one can derive the following bound for the difference of traces:  For $0< t <1$, there is a positive constant $C$ such that 
\begin{align*}
|\text{\rm HTr}K_{M_{q}}(z)+\text{\rm ETr}K_{M_{q}}(z)-\text{\rm DTr}K_{M_{q}}(z)|
\le Ct^{-2}(1+|s|)^{3/2}\,\,,
\end{align*}
holds (c.f. Remark 5.5 of \cite{GJ 16}.)

\hskip 0.2in As a consequence of Theorem \ref{pointwiseanduniformconvergence}, we have the following corollary, which describes the small time asymptotic behavior for the regularized trace of the heat kernel.

\begin{cor}\label{smalltasymptotics}
Let $M_{q}$ denote an elliptically degenerated family of
compact or non-compact hyperbolic Riemann surfaces of finite volume
which converges to the non-compact hyperbolic surface $M_{\infty}$.
Then for any fixed $\delta>0$, there exists a positive constant $c$
such that for all $0<t<\delta$, we have
\begin{align*}
\text{\rm HTr}K_{M_{q}}(t)+ \text{\rm ETr}K_{M_{q}}(t) - \text{\rm DTr}K_{M_{q}}(t) =
O\left(t^{-3/2}\right)\,\,
\end{align*}
uniformly in $q$.
\end{cor}

\hskip 0.2in Aside from the asymptotics  near $t=0$, we also need the behavior of the trace for large values of the time parameter $t$. In this direction, we need the following definition and theorem which we recall from Section 7 of \cite{GJ 16} where all details are presented. \begin{defn}\label{alphatruncatedtrace}
Let $M_{q}$ be an elliptically degenerating family of compact
or non-compact hyperbolic Riemann surfaces of finite volume which
converge to the non-compact hyperbolic surface $M_{\infty}$. Let
$0\le \alpha  < 1/4$ be such that $\alpha$ is not an eigenvalue of
$M_{\infty}$. We defined the $\alpha$-truncated hyperbolic and elliptic trace by
\begin{align*}
\text{\rm HTr}K_{M_{q}}^{(\alpha)}(t) + \text{\rm ETr}K_{M_{q}}^{(\alpha)}(t) =
\text{\rm HTr}K_{M_{q}}(t) + \text{\rm ETr}K_{M_{q}}(t) -
\sum_{\lambda_{q,n}\le \alpha} e^{-\lambda_{q,n}t}.
\end{align*}
\end{defn}

\begin{thm}\label{uniformlongtimeasymptotics}
Let $M_{q}$ be an elliptically degenerating family of compact
or non-compact hyperbolic Riemann surfaces of finite volume which
converge to the non-compact hyperbolic surface $M_{\infty}$. Let $\alpha$ be given according to  the Definition \ref{alphatruncatedtrace} above.
Then for any $c < \alpha$, there exist a constant $C$ such that the
bound
\begin{align*}
|\text{\rm HTr}K_{M_{q}}^{(\alpha)}(t) + \text{\rm ETr}K_{M_{q}}^{(\alpha)}(t) -
\text{\rm DTr}K_{M_{q}}(t)| \le C e^{-ct}
\end{align*}
holds for all $t\ge 0$ and uniformly in $q$.
\end{thm}

\section{Asymptotics of spectral measures}
\label{Asymptotics of spectral measures}

\hskip 0.2in We start the section with some general remarks on the
Laplace transforms of a function. This material can also be found in
\cite{HJL 97}. However, to make the reading self contained we
present the material below.

\hskip 0.20in For any function $f(t)$ defined on the positive real
line, we formally define the Laplace transform and cumulative
distribution function to be
\begin{align*}
\mathscr{L} (f)(z) = \int_{0}^{\infty} e^{-zt}f(t)dt \quad
\textrm{and} \quad F(t) = \int_{0}^{t}f(u)du.
\end{align*}
The Laplace transform $\mathscr{L}(f)$ of $f$ exists, if say $f(t)$
is a piecewise continuous, real-valued function for $0\le t <\infty$
and for some constants $M$ and $c$ we have that $|f(t)|\le Me^{a_0 t}$.
Then the Laplace transform will make sense in a some right half
plane $\textrm{Re}(z)>a_0$. The inversion formula for the Laplace transform allows us to write
\begin{align*}
f(t) = \frac{1}{2\pi i}\int_{a-i\infty}^{a+i\infty}
e^{tz}\mathscr{L}(f)(z)dz \quad \textrm{and} \quad F(t)
=\frac{1}{2\pi i}
\int_{a-i\infty}^{a+i\infty}e^{tz}\mathscr{L}(f)(z)\frac{dz}{z}.
\end{align*}
which holds for any $a>a_0$.

\begin{rmk}
\label{assumptionLaplacetransform}
\emph{
We will assume that $f$ is such that its Laplace
transform exists and the inversion formula holds. Furthermore, we
will need the following basic assumption
\begin{align*}
\int_{a-i\infty}^{a+i\infty}(1+|s|)^{3/2}|\mathscr{L}(f)(z)||e^{zT}|\frac{|dz|}{|z|}
<\infty
\end{align*}
where $z = t+is$ and $a$ is some positive number.
}
\end{rmk}

\hskip 0.2in As an application of the convergence of the regularized
trace of the heat kernel, we have the following theorem which is
analogous with Theorem 2.1 of \cite{HJL 97}. Note that the theorem
in \cite{HJL 97} deals with hyperbolic degeneration whereas the
theorem below deals with elliptic degeneration.

\begin{thm}\label{convergencespectralmeasure}
Let $M_{q}$ be an elliptically degenerating family of compact
or non-compact hyperbolic Riemann surfaces of finite volume
converging to the non-compact hyperbolic surface $M_{\infty}$. Let
$f$ be any function which satisfies the above assumption. Let $z = t+is$ with $t>0$ and denote by
\begin{align*}
N_{M_{q}}(f)(T) = \frac{1}{2\pi
i}\int_{a-i\infty}^{a+i\infty}\mathscr{L}(f)(z)\mathrm{Str}K_{M_{q}}(z)e^{zT}\frac{dz}{z}
\end{align*}
and
\begin{align*}
N_{M_{q},D}(f)(T) = \frac{1}{2\pi
i}\int_{a-i\infty}^{a+i\infty}\mathscr{L}(f)(z)\mathrm{Dtr}K_{M_{q}}(z)e^{zT}\frac{dz}{z}.
\end{align*}
Then
\begin{align*}
\lim_{q \to \infty}
[N_{M_{q}}(f)(T)-N_{M_{q},D}(f)(T)]
=N_{M_{\infty}}(f)(T)\,\,.
\end{align*}
\end{thm}
\begin{proof}
Consider the sequence of functions $g_{q}(z)$ given by
\begin{align*}
g_{q}(z) &= \mathscr{L}(f)(z) \left[ \textrm{Str}K_{M_{q}}(z)-\textrm{Dtr}K_{M_{q}}(z)\right] \frac{e^{zT}}{z}\\
g_{\infty}(z) &= \mathscr{L}(f)(z)  \textrm{Str}K_{M_{\infty}}(z) \frac{e^{zT}}{z}\,\,.
\end{align*}
We need to show that
\begin{align*}
\lim_{q\to \infty} \frac{1}{2\pi i} \int_{a-i\infty}^{a+i\infty} g_{q}(z) dz = \frac{1}{2\pi i} \int_{a-i\infty}^{a+i\infty} g_{\infty}(z)dz\,\,.
\end{align*}

As $q$ runs off to infinity, using part (a) of Theorem \ref{pointwiseanduniformconvergence}, $g_{q}(z)$ converges pointwise to $g_{\infty}(z)$ whenever $t = \text{Re\,\,}(z)>0.$ Using part (b) of the very same theorem, we also get that the functions are bounded uniformly, that is
\begin{align*}
|g_{q}(z)| \le |\mathscr{L}(f)(z)| (1+|s|)^{3/2} \left|\frac{e^{zT}}{z}\right|\,\,.
\end{align*}
Furthermore, the assumption on $f$ coming from the Remark \ref{assumptionLaplacetransform} requires that the right-hand side of the above inequality is integrable on vertical lines. All the hypotheses of the dominated convergence theorem are met, so that we can interchange the limit and the integration.
\end{proof}

\section{Convergence of spectral counting functions and small eigenvalues}
\label{Convergence of spectral counting functions and small eigenvalues}

\hskip 0.2in In this section, we will make use of the Theorem
\ref{convergencespectralmeasure} as applied to a particular family
of test functions which come from analytic number theory and
spectral theory. In this particular case, the functions mentioned in
Theorem \ref{convergencespectralmeasure} will be called spectral
weighted counting functions with parameter $w \ge 0$. For these
functions  and their associated degenerating component, we can explicitly determine the asymptotic behavior  for fixed $T> 0$ and all
$w \ge 0$.

\hskip 0.2in Consider the following family of
functions with parameter $w\ge 0$
\begin{align*}
f_w(t)=(w+1)t^w.
\end{align*}
It follows immediately that the corresponding Laplace transform and
cumulative distribution are given by
\begin{align*}
\mathscr{L}(f_w)(z) = \frac{\Gamma(w+2)}{z^{w+1}} \quad \textrm{and}
\quad F_w(t) = t^{w+1}
\end{align*}
respectively. With these remarks in mind, we can now define the
\emph{regularized} weighted spectral  counting function on a
hyperbolic Riemann surface $M$ by
\begin{align*}
N_{M,w+1}(T) = N_{M}(f_w(t))(T)=\frac{1}{2\pi
i}\int_{a-i\infty}^{a+i\infty}\frac{\Gamma(w+2)}{z^{w+1}}\textrm{Str}K_{M}(z)e^{zT}\frac{dz}{z}.
\end{align*}
In a similar fashion, we define the \emph{degenerating elliptic} weighted spectral
counting functions on the family $M_{q}$, by using $\textrm{Dtr}K_{M_{q}}(z)$ in lieu of $\textrm{Str}K_{M_{q}}(z).$
By Theorem \ref{pointwiseanduniformconvergence}, these weighted spectral counting functions are defined for values of the parameter
$w>3/2$.

\hskip 0.2in If the surface $M$ is compact, the regularized trace \emph{equals} the trace of the heat kernel (see the Remark \ref{Selbergtraceformula}). Using the spectral side of the Selberg trace formula (see equation (\ref{regularizedtracecompactspectral})) together with the mechanism of the inversion formula for the Laplace transforms, one can show that
\begin{align}\label{countingfunctioncompact}
N_{M,w}(T) = \sum_{\lambda_n \le T} (T-\lambda_n)^w\,\,.
\end{align}
In the non-compact case, the regularized trace \emph{equals} the trace of the heat kernel minus the contribution to the trace of the parabolic conjugacy classes. Using the spectral side of the trace as given by  equation (\ref{regularizedtracenoncompactspectral})together with the inversion formula, we obtain
\begin{align}\label{countingfunctionnoncompact}
\notag N_{M,w}(T) =& \sum_{\lambda_n\le T} (T-\lambda_n)^{w} -
\frac{1}{4\pi}\int_{-\sqrt{T-1/4}}^{\sqrt{T-1/4}}(T-1/4-r^2)^w
\frac{\phi'}{\phi}(1/2+ir)dr\\
\notag &+\frac{p}{2\pi}\int_{-\sqrt{T-1/4}}^{\sqrt{T-1/4}}(T-1/4-r^2)^w
\frac{\Gamma'}{\Gamma}(1+ir)dr\\
\notag &-\frac{1}{4}(p-\mathrm{Tr}\Phi(1/2))(T-1/4)^w\\
 &+\frac{p\log 2\,\,\Gamma(w+1)}{\sqrt{4\pi }\Gamma(w+3/2)}(T-1/4)^{w+1/2}\,\,,
\end{align}
whenever $T>1/4$, and
\begin{align*}
N_{M,w}(T) = \sum_{\lambda_n \le T} (T-\lambda_n)^w\
\end{align*}
if $T\le 1/4.$

\hskip 0.2in As a direct application of Theorem
\ref{convergencespectralmeasure} we have the following result.

\begin{thm}\label{convergencecountingfunctions}
Let $M_{q}$ denote an elliptically degenerating family of
compact or non-compact hyperbolic Riemann surfaces of finite volume
converging to the non-compact hyperbolic surface $M_{\infty}.$ For
any $w>3/2$ define
\begin{align*}
G_{M_{q},w}(T)= N_{M_{q},D}(f_{w-1}(t))(T)=\frac{1}{2\pi
i}\int_{a-i\infty}^{a+i\infty}\frac{\Gamma(w+1)}{z^{w}}\mathrm{Dtr}K_{M_{q}}(z)e^{zT}\frac{dz}{z}.
\end{align*}
Then for $T>0$ we have that
\begin{align*}
\lim_{q \to \infty}
[N_{M_{q},w}(T)-G_{M_{q},w}(T)]= N_{M_{\infty},w}(T)\,\,.
\end{align*}
\end{thm}

\hskip 0.2in The next result establishes the asymptotic behavior of
the function $G_{M_{q},w}(T)$ for fixed $T>1/4$ and weight $w
\ge 0$.
\begin{prop}\label{degeneratingellipticcountingprop}
Let $M_{q}$ denote an elliptically degenerating family of
compact or non-compact hyperbolic Riemann surfaces of finite volume
converging to the non-compact hyperbolic surface $M_{\infty}.$ For
any degenerating elliptic representative $\gamma \in
DE(\Gamma_{q})$ let $q_{\gamma}$ denote the order of the
corresponding finite cyclic subgroup.
\begin{enumerate}[(a)]
\item For any $w\ge 0$ and $T > 1/4$ we have
\begin{align*}
G_{M_{q},w}(T) = \sum_{\gamma \in DE(\Gamma_{q})}
\sum_{n=1}^{q_{\gamma}-1}
\frac{1}{2q_{\gamma}\sin(n\pi/q_{\gamma})}
\int\limits_{-\sqrt{T-1/4}}^{\sqrt{T-1/4}}
(T-1/4-r^2)^w\frac{e^{-2\pi n r/q_{\gamma}}}{1+e^{-2\pi r}}dr.
\end{align*}
\item For any $w\ge 0$ and $T\le 1/4$ we have $G_{M_{q},w}(T) = 0$, independently of $q$.
\item
For fixed $w\ge 0$
and $T>1/4$ we have
\begin{align*}
G_{M_{q},w}(T) = \frac{1}{\pi} \log\left(\prod_{\gamma \in DE(\Gamma)} q_{\gamma}\right)
\int\limits_{-\sqrt{T-1/4}}^{\sqrt{T-1/4}}
(T-1/4-r^2)^w\frac{1}{1+e^{-2\pi r}}dr + O(1)
\end{align*}
as the $q_\gamma$'s tend to infinity.
\end{enumerate}
\end{prop}

\begin{proof}
 We are studying the inverse Laplace transform of
\begin{align*}
\textrm{Dtr}K_{M_{q}}(t) = \sum_{\gamma \in DE(\Gamma_{q})}
\sum_{n=1}^{q_{\gamma}-1}
\frac{e^{-t/4}}{2q_{\gamma}\sin(n\pi/q_{\gamma})} \int_{-\infty}^{\infty}
\frac{e^{-2\pi n r/q_{\gamma}-tr^2}}{1+e^{-2\pi r}}dr.
\end{align*}
Using the definition of the degenerating elliptic spectral counting function together with the mechanism of the Laplace inversion formula allows us to write
\begin{align*}
G_{M_{q},w}(T) = \sum_{\gamma \in DE(\Gamma_{q})}
\sum_{n=1}^{q_{\gamma}-1}
\frac{1}{2q_{\gamma}\sin(n\pi/q_{\gamma})}
\int\limits_{-\sqrt{T-1/4}}^{\sqrt{T-1/4}}
(T-1/4-r^2)^w\frac{e^{-2\pi n r/q_{\gamma}}}{1+e^{-2\pi r}}dr
\end{align*}
provided that $T>1/4$.   In the case $T\le 1/4$, the properties of inverse Laplace transform imply that the integral over the vertical line equals zero, hence $G_{M_{q},w}(T) = 0$. Recall that, from the definition of the weighted spectral counting function, we know that such functions are only defined for weights $w > 3/2.$ However, the above formula is defined for any $w\ge 0$. This in turn, allows us to extend the definition of the degenerating (as well as elliptic) weighted spectral counting to any non-negative weights $w$.  This proves parts (a) and (b) of the theorem.

\hskip 0.2in To prove part (c), we start by fixing $T > 1/4$. We note that
\begin{align*}
e^{-2\pi n r/q_{\gamma}} = 1 + O(r/q_{\gamma}) \textrm{ for } r^2\le
T-1/4,
\end{align*}
so then
\begin{align}\label{degeneratingcountingfunction}
\notag G_{M_{q},w}(T) =& \sum_{\gamma \in DE(\Gamma_{q})}
\sum_{n=1}^{q_{\gamma}-1} \frac{1}{2q_{\gamma}\sin(n\pi/q_{\gamma})}
\int\limits_{-\sqrt{T-1/4}}^{\sqrt{T-1/4}}
(T-1/4-r^2)^w\frac{1}{1+e^{-2\pi r}}dr\\
&+\sum_{\gamma \in DE(\Gamma_{q})} \sum_{n=1}^{q_{\gamma}-1}
\frac{1}{2q_{\gamma}^2\sin(n\pi/q_{\gamma})}
\int\limits_{-\sqrt{T-1/4}}^{\sqrt{T-1/4}}
(T-1/4-r^2)^w\frac{O(r)}{1+e^{-2\pi r}}dr.
\end{align}
To continue, we focus on estimating the sum
\begin{align*}
S(q_{\gamma})=\sum_{n=1}^{q_{\gamma}-1}
\frac{1}{2q_{\gamma}\sin(n\pi/q_{\gamma})}
\end{align*}
as $q_{\gamma} \to \infty$, since such an estimate would apply to
estimate the function $G_{M_{q},w}(T)$.

Let us write
\begin{align*}
S(q_{\gamma})=\sum_{n=1}^{[q_{\gamma}/4]}
\frac{1}{2q_{\gamma}\sin(n\pi/q_{\gamma})} +
\sum_{n=[q_{\gamma}/4]+1}^{[3q_{\gamma}/4]}
\frac{1}{2q_{\gamma}\sin(n\pi/q_{\gamma})}+
\sum_{n=[3q_{\gamma}/4]+1}^{q_{\gamma}-1}
\frac{1}{2q_{\gamma}\sin(n\pi/q_{\gamma})}.
\end{align*}
We recognize the middle sum as a Riemann sum. As such we can write
its limiting value as
\begin{align*}
\sum_{n=[q_{\gamma}/4]+1}^{[3q_{\gamma}/4]}
\frac{1}{2q_{\gamma}\sin(n\pi/q_{\gamma})} \to
\frac{1}{2\pi}\int\limits_{\pi/4}^{3\pi/4} \frac{dx}{\sin x}  = O(1) \quad
\textrm{ as } q_{\gamma}\to \infty.
\end{align*}
Using the identity $\sin(x)=\sin(\pi - x)$, we then have that
\begin{align*}
S(q_{\gamma})=\sum_{n=1}^{[q_{\gamma}/4]}
\frac{1}{q_{\gamma}\sin(n\pi/q_{\gamma})} + O(1) \quad \textrm{ as
} q_{\gamma} \to \infty.
\end{align*}
For $x\in [0,\pi/4]$, we have that $x-x^3/6\le \sin x \le x$, so
then
\begin{align*}
\frac{1}{x} \le \frac{1}{\sin x} \le \frac{1}{x-x^3/6} \quad
\textrm{ for } x\in[0,\pi/4]
\end{align*}
which further implies
\begin{align*}
0 \le \frac{1}{\sin x} - \frac{1}{x} \le \frac{1}{x-x^3/6}
-\frac{1}{x} = \frac{x}{6(1-x^2/6)} \quad \textrm{ for }
x\in[0,\pi/4].
\end{align*}
With all this, we take $x=n\pi/q_{\gamma}$ with $1\le n \le
[q_{\gamma}/4]$ an arrive at the bounds
\begin{align*}
0 \le \frac{1}{q_{\gamma}}
\sum_{n=1}^{[q_{\gamma}/4]}\Bigg(\frac{1}{\sin(n\pi/q_{\gamma})}-
\frac{1}{n\pi/q_{\gamma}}\Bigg) \le \frac{1}{q_{\gamma}}
\sum_{n=1}^{[q_{\gamma}/4]}
\frac{n\pi/q_{\gamma}}{6(1-(n\pi/q_{\gamma})^2/6)}.
\end{align*}
This upper sum is also recognizable as a Riemann sum, so then we can
write
\begin{align*}
\frac{1}{q_{\gamma}} \sum_{n=1}^{[q_{\gamma}/4]}
\frac{n\pi/q_{\gamma}}{6(1-(n\pi/q_{\gamma})^2/6)} \to \frac{1}{\pi}
\int\limits_{0}^{\pi/4}\frac{x}{6(1-x^2/6)}dx \quad \textrm{ as }
q_{\gamma} \to \infty.
\end{align*}
The above integral is clearly finite. Therefore, we have shown that
\begin{align*}
S(q_{\gamma}) - \frac{1}{\pi}\sum_{n=1}^{[q_{\gamma}/4]} \frac{1}{n}
= O(1) \quad \textrm{ as } q_{\gamma} \to \infty.
\end{align*}
It is elementary to show that
\begin{align*}
\sum_{n=1}^{[q_{\gamma}/4]} \frac{1}{n} = \log q_{\gamma}+O(1) \quad
\textrm{ as } q_{\gamma} \to \infty.
\end{align*}

Thus the first inner sum in the right-hand side of equation (\ref{degeneratingcountingfunction}) has the asymptotic
\begin{align*}
S(q_{\gamma}) = \frac{1}{\pi} \log(q_{\gamma}) + O(1) \quad \textrm{ as } q_{\gamma} \to \infty\,\,.
\end{align*}
Consequently, the second inner sum in the right-hand side of (\ref{degeneratingcountingfunction}), namely $q_{\gamma}^{-1}S(q_{\gamma})$ approaches zero as $q_{\gamma}$ runs off to infinity.
Applying these estimates to equation (\ref{degeneratingcountingfunction}) completes the proof.
\end{proof}

\hskip 0.2in Our next task is to study the behavior of the weighted spectral counting functions $M_{M_{q},w}(T)$ for weights $0\le w\le 3/2$ in both compact and non-compact cases. We start by making the following observations coming from Proposition \ref{degeneratingellipticcountingprop}. Consider the integral in the formula for $G_{M_{q},w}(T)$
\begin{align*}
c_w(T)= \frac{1}{\pi}\int\limits_{-\sqrt{T-1/4}}^{\sqrt{T-1/4}}
(T-1/4-r^2)^w\frac{e^{-2\pi n r/q_{\gamma}}}{1+e^{-2\pi r}}dr\,\,.
\end{align*}
Let $f(T,r)$ denote the integrand above. Since $f(T,r)$ as well as the limits of integration are $C^1$ in both variables we have that
\begin{align}\label{cwTrelation}
\notag \frac{d}{dT} c_{w+1}(T) =&  f\left(T, \sqrt{T-1/4}\right) \frac{d}{dT} \left(\sqrt{T-1/4}\right) -  f\left(T, -\sqrt{T-1/4}\right) \frac{d}{dT} \left(-\sqrt{T-1/4}\right)\\
\notag &+ \frac{1}{\pi}\int\limits_{-\sqrt{T-1/4}}^{\sqrt{T-1/4}}
\frac{d}{dT} \left[(T-1/4-r^2)^{(w+1)}\frac{e^{-2\pi n r/q_{\gamma}}}{1+e^{-2\pi r}}\right]dr\\
=& (w+1) c_w(T)\,\,,
\end{align}
for any $w\ge 0$. Setting $Q = \log(\prod q_{\gamma})$, where the product runs over all the degenerating elliptic elements of $\Gamma_{q}$, we can write
\begin{align}\label{asymptoticelldegencountingfunction}
G_{M_{q},w}(T) = c_w(T) \log Q
+ O(1)
\end{align}
as the $q$ tends to infinity. Furthermore, in the special case $w=0$, we can apply the mean value theorem to estimate the integral the defines $c_0(T)$. Namely, for some value $c$ in the domain of integration, we get
\begin{align*}
c_0(T)=\frac{e^{-2\pi n c/q_{\gamma}}}{1+e^{-2\pi c}} \cdot \frac{2 \sqrt{T-1/4}}{\pi}\,\,.
\end{align*}
This allows to rewrite the behavior of the weight 0 degenerating elliptic counting function as
\begin{align*}
G_{M_{q},0}(T) = \frac{2  C \sqrt{T-1/4}}{\pi} \log Q
+ O(1) \,\,,
\end{align*}
as $q$ tends to infinity and for some $0<C<1$.

\hskip 0.2in We continue by making the following observation. For $w>1/2$, the expression for the weighted counting function associated to the compact family $M_w$ as given by (\ref{countingfunctioncompact}) implies
\begin{align*}
\frac{1}{w+1} \cdot \frac{d}{dT} N_{M_{q},w+1}(T) = \sum_{\lambda_{n,q} < T} (T-\lambda_{n,q})^w\,\,.
\end{align*}
The left-hand side above is defined since, $w+1>3/2$. It is also clear that the right-hand side above is a well defined function.
This allows us to define $N_{M_{q},w}(T)$ for values of the weight above $1/2$, namely,
\begin{align}\label{countingfunctiondiffeqn}
N_{M_{q},w}(T) = \frac{1}{w+1} \cdot \frac{d}{dT} N_{M_{q},w+1}(T) \,\,.
\end{align}
By repeating the above argument, we can extend $N_{M_{q},w}(T)$ to any $w\ge 0$. In particular, $N_{M_{q},0}(T) $ counts with multiplicity the eigenvalues of the Laplace operator on $M_{q}$ which are less than $T$. With the above remarks in mind, we are now ready to give the behavior of the counting function $N_{M_{q},w}(T)$ for any weight $0\le w\le 3/2$ in the compact case.

\begin{thm}\label{behaviorcountingfunctionscompact}
Let $M_{q}$ denote an elliptically degenerating family of
compact hyperbolic Riemann surfaces of finite volume.
Then for $T > 1/4$ and $0\le w \le 3/2$ we have that
\begin{align*}
N_{M_{q},w}(T) \sim  c_w(T) \log Q
\end{align*}
as $q$ tends to infinity.
\end{thm}
\begin{proof}
Given any $w\ge 0$,  the counting function $N_{M_{q},w}(T) $ is increasing for $T>0$. Choose any $\varepsilon > 0$. The mean value theorem applied to $N_{M_{q},w}(T)$ on the interval $[T,T+\varepsilon]$ together with the differential equation satisfied by the counting functions (see formula (\ref{countingfunctiondiffeqn})) as well as the monotonicity imply
\begin{align}\label{countingfunctionMVT}
N_{M_{q},w}(T) \le \frac{1}{w+1} \frac{N_{M_{q},w+1}(T+\varepsilon) - N_{M_{q},w+1}(T)}{\varepsilon} \le N_{M_{q},w}(T+\varepsilon)\,\,.
\end{align}
Now fix a weight $w>1/2$. Then we can write the inequalities in (\ref{countingfunctionMVT}) above as
\begin{align}\label{countingfunctionMVT2}
\frac{N_{M_{q},w}(T)}{\log Q} \le \frac{1}{w+1} \frac{N_{M_{q},w+1}(T+\varepsilon)/\log Q - N_{M_{q},w+1}(T)/\log Q}{\varepsilon} \le \frac{N_{M_{q},w}(T+\varepsilon)}{ \log Q}\,\,.
\end{align}
Taking the limit as $q$ goes to infinity in (\ref{countingfunctionMVT2} ), together with the convergence of counting functions of weight $w>3/2$ (see Theorem \ref{convergencecountingfunctions}) and the asymptotic coming from (\ref{asymptoticelldegencountingfunction}) applied to the middle term imply that
\begin{align}\label{countingfunctionMVT3}
\limsup_{q \to \infty }\frac{N_{M_{q},w}(T)}{ \log Q} \le \frac{1}{w+1} \frac{c_{w+1}(T+\varepsilon) - c_{w+1}(T)}{\varepsilon} \le \liminf_{q \to \infty }\frac{N_{M_{q},w}(T+\varepsilon)}{\log Q}\,\,.
\end{align}
Letting $\varepsilon$ go to zero and using the differential equation satisfied by $c_{w+1}(T)$ (coming from (\ref{cwTrelation})), to obtain
\begin{align}\label{countingfunctionMVT4}
\limsup_{q \to \infty }\frac{N_{M_{q},w}(T)}{ \log Q} \le c_{w}(T) \le \liminf_{q \to \infty }\frac{N_{M_{q},w}(T)}{ \log Q}\,\,.
\end{align}
This proves that for weights $w>1/2$ we have
\begin{align*}
\lim_{q \to \infty }\frac{N_{M_{q},w}(T)}{ \log Q} = c_{w}(T)\,\,.
\end{align*}
Fix $w\ge 0$ and repeat the above argument to extend the result to any non-negative weight $w.$
\end{proof}

\hskip 0.2in Let us continue by investigating the behavior of the counting functions for weights $0\le w\le 3/2$ associated to the non-compact family $M_{q}$.
\begin{thm}\label{behaviorcountingfunctionsnoncompact}
Let $M_{q}$ denote an elliptically degenerating family of
non-compact hyperbolic Riemann surfaces of finite volume.
Then for $T > 1/4$ and $0\le w \le 3/2$ we have that
\begin{align*}
N_{M_{q},w}(T) \sim  c_w(T) \log Q
\end{align*}
as $q$ tends to infinity.
\end{thm}
\begin{proof}
We need to show that for fixed $T>1/4$ and $0\le w \le 3/2$ the following limit holds
\begin{align}\label{behaviorcountingfunctionsnoncompactlimit}
\lim_{q \to \infty }\frac{N_{M_{q},w}(T)}{\log Q}  = c_w(T) \,\,.
\end{align}
\hskip 0.2in Recall that for $T>1/4$ and $w>3/2$, the spectral counting function $N_{M_{q},w}(T)$ is given by formula (\ref{countingfunctionnoncompact}). Let us look at the 5 terms that amount the counting function.  For the third term we have that
\begin{align*}
\left|\frac{\Gamma'}{\Gamma}(1+ir)\right| \le \left|\frac{\Gamma'}{\Gamma}(1)\right| = \Gamma'(1) = \gamma
\end{align*}
where $\gamma$ denotes the Euler-Mascheroni constant (coming from p. 114 \cite{JoLa 93b} ). This shows that this term in the expression of the spectral counting function is finite and independent of $q$.  Consequently, the contribution of this term to the limit (\ref{behaviorcountingfunctionsnoncompactlimit}) is zero.
The fourth term in (\ref{countingfunctionnoncompact}) involves the trace of the scattering matrix at $s=1/2$. From \cite{Ku 73}, we have that the $p\times p$ matrix $A=\Phi(1/2)$ is  orthogonal and symmetric. Then $A^2 = \text{Id}$ which implies that the only eigenvalues of the matrix $A$ are $\pm 1$. Since the trace of the matrix equals the sum of its eigenvalues, it follows that $|\text{Tr }\Phi(1/2)| \le p$. Consequently, the fourth term in (\ref{countingfunctionnoncompact}) is bounded independently of $q$, so that its contribution to the limit (\ref{behaviorcountingfunctionsnoncompactlimit}) is zero. The contribution of the fifth term of the spectral counting function to the above limit is clearly zero.

\hskip 0.2in So far we have shown that only the first two terms in the right-hand side of equation (\ref{countingfunctionnoncompact}) have a significant contribution to the spectral counting function. To this end, let us define
\begin{align}\label{countingfunctionnoncompact2}
\widehat{N_{M,w}(T)} = \sum_{\lambda_n\le T} (T-\lambda_n)^{w} -
\frac{1}{4\pi}\int_{-\sqrt{T-1/4}}^{\sqrt{T-1/4}}(T-1/4-r^2)^w
\frac{\phi'}{\phi}(1/2+ir)dr\,\,.
\end{align}
By the previous remarks, it remains to show that
\begin{align}\label{behaviorcountingfunctionsnoncompactlimit2}
\lim_{q \to \infty }\frac{\widehat{N_{M_{q},w}(T)}}{\log Q}  = c_w(T) \,\,.
\end{align}

Quoting Lemma 5.3 of \cite{HJL 97} (which comes from pp. 160 of \cite{He 83}), we have the following result
\begin{align*}
-\frac{\phi'}{\phi}(1/2 +ir) - \sum_{k=1}^{N} \frac{1-s_{k,q}}{(s_{k,q}-1/2)^2 + r^2} \ge 2 \log q_{M_{q}} > 0\,\,,
\end{align*}
where $1/2 < s_{k,q}\le 1$ and $q_{M_{q}}>1.$ This allows to write
\begin{align}\label{countingfunctionnoncompact3}
\notag \widehat{N_{M,w}(T)} = \sum_{\lambda_n\le T} (T-& \lambda_n)^{w} \\
\notag -\frac{1}{4\pi}\int_{-\sqrt{T-1/4}}^{\sqrt{T-1/4}}(T-1/4-& r^2)^w
\left(\frac{\phi'}{\phi}(1/2+ir) +  \sum_{k=1}^{N} \frac{1-s_{k,q}}{(s_{k,q}-1/2)^2 + r^2} \right)dr\\
+ \frac{1}{4\pi}\int_{-\sqrt{T-1/4}}^{\sqrt{T-1/4}}(T-1/4-& r^2)^w
\sum_{k=1}^{N} \frac{1-s_{k,q}}{(s_{k,q}-1/2)^2 + r^2} dr\,\,.
\end{align}
Consequently, the hat spectral counting function, as given by (\ref{countingfunctionnoncompact3}), is increasing whenever $w\ge 0$ and $T>0$. Furthermore, the hat function satisfies the differential equation as in (\ref{countingfunctiondiffeqn}). For $w>3/2$, the result of the Theorem \ref{convergencecountingfunctions} applies. Fix a weight $w>1/2$ and proceed as in (\ref{countingfunctionMVT}) through (\ref{countingfunctionMVT4}) to show that
\begin{align*}
\lim_{q \to \infty }\frac{\widehat{N_{M_{q},w}(T)}}{\log Q}  = c_w(T) \,\,.
\end{align*}
Repeating the argument, but now with $w\ge 0$ fixed, completes the proof.
\end{proof}

\hskip 0.2in As an immediate consequence of Theorem \ref{convergencecountingfunctions} and Proposition \ref{degeneratingellipticcountingprop} together with the fact that these counting functions extend to any non-negative weight, we obtain the following corollary.

\begin{cor} \label{convergencecountingfunctionssmallT}
Let $M_{q}$ denote an elliptically degenerating family of
compact or non-compact hyperbolic Riemann surfaces of finite volume
converging to the non-compact hyperbolic surface $M_{\infty}.$
Then for $T\le 1/4$ and $w>0$ we have that
\begin{align*}
\lim_{q \to \infty}
N_{M_{q},w}(T)= N_{M_{\infty},w}(T).
\end{align*}
In addition to this, if $T$ is not an eigenvalue of $M_{\infty}$, we get that
\begin{align*}
\lim_{q \to \infty}
N_{M_{q},0}(T)= N_{M_{\infty},0}(T).
\end{align*}
\end{cor}

\hskip 0.2in In the case $T\le 1/4$, the weighted spectral counting functions for $M_{q}$ in both compact and non-compact case (see equations (\ref{countingfunctioncompact}) and (\ref{countingfunctionnoncompact})) turn out to be
\begin{align*}
N_{M_{q},w}(T) = \sum_{\lambda_{n,q} < T} (T-\lambda_{n,q})^{w} \,\,.
\end{align*}
The above corollary states that convergence through elliptic degeneration of functions that involve the eigenvalues below 1/4.
Furthermore, the above corollary in implies the convergence of these small eigenvalues through elliptic degeneration. In particular, if the eigenvalue $\lambda_{n,q}$ has multiplicity one, then we have
\begin{align*}
\lim_{q \to \infty} \lambda_{n,q} = \lambda_{n,\infty}\,\,.
\end{align*}

\begin{rmk}\label{countingfunctionerrorterm}
\emph{
We note that Theorems \ref{behaviorcountingfunctionscompact} and \ref{behaviorcountingfunctionsnoncompact} present the asymptotic behavior of the counting function $N_{M_{q},w}(T)$ for $T>1/4$ and weights $0\le w\le 3/2$, in both the compact and non-compact case. These two results only mention the behavior of the leading term and nothing about the error term. Modifications in the course of the proof of the two theorems can lead to results about the error term. To get such results, one needs to assume something extra about the rate at which $\varepsilon$ approaches zero. More precisely, $\varepsilon$ should approach zero at a rate that depends on the degenerating parameter $q$. A similar situation had been studied in \cite{HJL 97} in the context of hyperbolic degeneration. \\
\text{\hskip 0.2in} From Theorem \ref{convergencecountingfunctions}, we have that for $w>1/2$, $T>1/4$, and arbitrarily large values of the degenerating parameter $q$
\begin{align*}
N_{M_{q},w+1}(T) = N_{M_{\infty},w+1}(T)+G_{M_{q},w+1}(T)+O(f(q))\,\,,
\end{align*}
for some function $f(q)$ which approaches zero when $q$ runs off to infinity. Choose $\varepsilon(q)>0$. Applying the mean value theorem on the interval $[T,T+\varepsilon(q)]$ allows us to write
\begin{align*}
N_{M_{q},w}(T) \le& \frac{1}{w+1} \frac{N_{M_{\infty},w+1}(T+\varepsilon(q)) - N_{M_{\infty},w+1}(T)}{\varepsilon(q)} \\
&+ \frac{1}{w+1}\frac{G_{M_{q},w+1}(T+\varepsilon(q)) - G_{M_{q},w+1}(T)}{\varepsilon(q)}
+ O\left(\frac{f(q)}{\varepsilon(q)}\right)\,\,.
\end{align*}
Using a linear approximation for the first two terms in the middle of the above inequality gives
\begin{align*}
N_{M_{q},w}(T) \le& N_{M_{\infty},w}(T) +  {\varepsilon(q)} \frac{d}{dT} N_{M_{\infty},w}(T_1) \\
&+ G_{M_{q},w}(T) +  {\varepsilon(q)} \frac{d}{dT} G_{M_{q},w}(T_2)
+ O\left(\frac{f(q)}{\varepsilon(q)}\right)\,\,,
\end{align*}
for some $T_1, T_2 \in [T,T+\varepsilon(q)].$ 
In a similar fashion, by applying the mean value theorem this time on the interval $[T-\varepsilon(q),T]$, it follows that
\begin{align*}
N_{M_{q},w}(T) \ge& N_{M_{\infty},w}(T) +  {\varepsilon(q)} \frac{d}{dT} N_{M_{\infty},w}(T_3) \\
&+ G_{M_{q},w}(T) +  {\varepsilon(q)} \frac{d}{dT} G_{M_{q},w}(T_4)
+ O\left(\frac{f(q)}{\varepsilon(q)}\right)\,\,,
\end{align*}
for some $T_3, T_4 \in [T-\varepsilon(q),T].$ 
Theorems \ref{behaviorcountingfunctionscompact}  and \ref{behaviorcountingfunctionsnoncompact} applied to the derivative terms imply the following asymptotic formula
\begin{align*}
N_{M_{q},w}(T) = N_{M_{\infty},w}(T) + G_{M_{q},w}(T) + O\left({\varepsilon(q)} \log Q\right)
+ O\left(\frac{f(q)}{\varepsilon(q)}\right)\,\,.
\end{align*}
One needs to optimize the way in which $\varepsilon(q)$ approaches zero so that the amount of error is minimized, namely by setting
$\varepsilon(q) = \sqrt{f(q)/\log Q}\,\,.$
 Optimizing the error in the case $w>1/2$ allows then for the improvement of the error in the case $w\ge 0.$
}
\end{rmk}

\hskip 0.2in Using the above remark together with some calculations that had been pointed out by Dennis Hejhal, we have the following result.
\begin{thm}\label{countingfunctionwitherror}
Let $M_{q}$ denote an elliptically degenerating family of
compact or non-compact hyperbolic Riemann surfaces of finite volume
converging to the non-compact hyperbolic surface $M_{\infty}.$
Then
\begin{align*}
N_{M_{q},0}(T)= c_0(T) \log Q + O\left( (\log Q)^{3/4} \right)\,\,.
\end{align*}
\end{thm}

\begin{proof}
The proof uses two applications of Remark \ref{countingfunctionerrorterm}. 
In the first step we set $w=1$. Following the computations of Proposition \ref{degeneratingellipticcountingprop}, we can take $f(q) = 1$. 
In this case, Remark \ref{countingfunctionerrorterm} begins with
\begin{align*}
N_{M_{q},2}(T) = N_{M_{\infty},2}(T) + G_{M_{q},2}(T) + O(1)
\end{align*}
and ends with
\begin{align*}
N_{M_{q},1}(T) = N_{M_{\infty},1}(T) + G_{M_{q},1}(T) + O\left({\varepsilon(q)} \log Q\right)
+ O\left(\frac{1}{\varepsilon(q)}\right)\,\,.
\end{align*}
Minimizing the error term implies $\varepsilon(q) = \left(\log Q\right)^{-1/2}\,\,$.

In the second step, Remark \ref{countingfunctionerrorterm} starts with 
\begin{align*}
N_{M_{q},1}(T) = N_{M_{\infty},1}(T) + G_{M_{q},1}(T) + O\left((\log Q)^{1/2}\right)
\end{align*}
and ends with
\begin{align*}
N_{M_{q},0}(T) = N_{M_{\infty},0}(T) + G_{M_{q},0}(T) + O\left({\varepsilon(q)} \log Q\right)
+ O\left(\frac{(\log Q)^{1/2}}{\varepsilon(q)}\right)\,\,.
\end{align*}
Minimizing the error term implies $\varepsilon(q) = \left(\log Q\right)^{-1/4}\,\,$. Consequently,
\begin{align*}
N_{M_{q},0}(T) = N_{M_{\infty},0}(T) + G_{M_{q},0}(T) + O\left( (\log Q)^{3/4}\right)\,\,.
\end{align*}
By formula (\ref{asymptoticelldegencountingfunction}) together with Theorems \ref{behaviorcountingfunctionscompact}  and \ref{behaviorcountingfunctionsnoncompact}, the first two terms on the right-hand side above grow like $c_0(T)\log Q\,\,.$
\end{proof}

\begin{example}
\emph{\textbf{Hecke triangle groups} Let $G_N$ be the Hecke
triangle group, which is the discrete group generated by
\begin{equation*}
z \mapsto -1/z \textrm{ and } z  \mapsto z + 2 \cos(\pi/N)
\end{equation*}
for any integer $N \ge 3$. The group is commensurable with
$\textrm{PSL}(2,\mathbb{Z})$ if and only if in the three cases when
$N = 3, 4, 6$. In all other cases, the non-arithmetic nature of
$G_N$ is such that certain precise, theoretical
computations can be impossible. However, the explicit nature of the
group theoretic definition of $G_N$ is such that numerical
methods can be employed (see, for example, \cite{He 92}). It can be shown that for each $N$, the quotient space
$G_N\backslash\mathbb{H}$ has genus zero with one cusp and
three elliptic points of order 2, 3, and $N$ (see \cite{He 83},
\cite{He 92}, and references therein). As such, the results in the
present paper apply. Specifically, Theorem \ref{countingfunctionwitherror} determines the
accumulation of the spectral densities as a function of $N$, a
result which is attributed to Selberg (see p. 579 of \cite{He 83}).
In other words, Theorem \ref{countingfunctionwitherror} above can be viewed as providing precise
quantification to Selberg's result.  }
\end{example}

\section{Spectral functions}
\label{Spectral functions}

\hskip 0.2in In this section, we investigate the behavior through degeneration of the spectral zeta function and Hurwitz zeta function, the former being a special case of the latter.
After we recall definitions, we present the analytic properties these functions posses as well as describe their behavior on a family of elliptically degenerating surfaces. The main ingredient in the process is the analysis of the various integral transforms of the trace of the heat kernel that realize these spectral functions.

\subsection{Spectral zeta function}

\hskip 0.2in Let us assume first that the surface $M$ is compact (with 1 connected component.) In this case, the spectral theory of the Laplacian asserts the spectrum $0 = \lambda_0 < \lambda_1 \le \lambda_2 \le \ldots \nearrow \infty$ counted with multiplicity. The positive eigenvalues can be used to form a Dirichlet series, the spectral zeta function $\zeta_{M}(s)$ which is defined by
\begin{align*}
\zeta_{M}(s) = \sum_{\lambda_{n}>0} \lambda_{n}^{-s}\,\,.
\end{align*}
By Weyl's law (\ref{Weyl'slaw}), the series converges absolutely and uniformly for $\mathrm{Re}(s) \ge 1 + \varepsilon$ with $\varepsilon >0\,\,.$ Hence $\zeta_{M}(s)$ is an analytic function for $\mathrm{Re}(s)>1\,\,.$

\hskip 0.2in Furthermore, we have
\begin{align*}
\Gamma(s) \zeta_{M}(s) &= \sum_{\lambda_{n}>0} \lambda_{n}^{-s} \Gamma(s)\\
&=\sum_{\lambda_{n}>0} \lambda_{n}^{-s} \int_{0}^{\infty} e^{-t} t^s \frac{dt}{t}\\ 
&= \int_{0}^{\infty} \sum_{\lambda_{n} >0} e^{-\lambda_{n}t} t^s \frac{dt}{t}\\ 
&=\int_{0}^{\infty}[\textrm{Str}K_M(t) - 1] t^s\frac{dt}{t}\,\,.
\end{align*}
The behavior of $\textrm{Str}K_M(t)$ at $t=0$ and $t=\infty$ (see (\ref{standardtraceasymp0}) and (\ref{standardtraceasympinfty}) respectively) shows that the above integral is defined for $\mathrm{Re}(s) > 1$.
Parenthetically, if the surface had $c_M$ connected components, then the value $1$ (c.f. the zero eigenvalue) in the above integrand would be replaced by $c_M$. For simplicity of notation, we assume that $c_M = 1$. That said, the above manipulations show that the spectral zeta function is essentially the Mellin transform of the
standard trace of the heat kernel, namely 
\begin{align}\label{spectralzetadefinition}
\zeta_{M}(s) 
&=\frac{1}{\Gamma(s)}\int_{0}^{\infty}[\textrm{Str}K_M(t) - 1]t^s\frac{dt}{t}\,\,.
\end{align}

\begin{prop}\label{spectralzetaprop}
Suppose that $M$ is a compact hyperbolic Riemann surface. The spectral zeta function $\zeta_{M}(s)$ has meromorphic continuation
to all $s \in \mathbb{C}$, except for a simple pole at $s = 1$ with residue $\mathrm{vol}(M)/(4\pi)\,\,.$
\end{prop}

\begin{proof}
The proof follows from the analysis of the integral representation of the spectral zeta from (\ref{spectralzetadefinition}) above. 
We start by breaking off the integral as follows 
\begin{align}\label{spectralzetacontinuation1}
\notag\zeta_{M}(s)  & =  \frac{1}{\Gamma(s)}\int_{0}^{1} \textrm{Str}K_M(t)t^{s-1}dt - \frac{1}{\Gamma(s+1)} +
\frac{1}{\Gamma(s)} \int_{1}^{\infty} [\textrm{Str}K_M(t) - 1]t^{s-1}dt \\
& =  G(s) - \frac{1}{\Gamma(s+1)} +
\frac{1}{\Gamma(s)} \int_{1}^{\infty} [\textrm{Str}K_M(t) - 1]t^{s-1}dt \,\,.
\end{align}
The second term above is entire. Since $\textrm{Str}K_M(t) - 1$ has exponential decay at infinity (c.f (\ref{standardtraceasympinfty})), the third term is also analytic. So we only need to continue the term containing the integral over $[0,1]$, which we call $G(s)$. For the latter, we recall (\ref{standardtraceasympinfty}), namely at $t=0$
\begin{align}\label{spectralzetacontinuation2}
\textrm{Str}K_M(t) = \frac{b_{-1}}{t} + b_0 + b_1t + b_2t^2 + \ldots
\end{align}
where for simplicity we use $b_{-1}$ in place of $\textrm{vol}(M)/(4\pi)$. That said, we can write for $\mathrm{Re}(s)>1$
\begin{align}\label{spectralzetacontinuation3}
\notag G(s)   & =  \frac{1}{\Gamma(s)} \int_{0}^{1} \textrm{Str}K_M(t) t^{s-1}dt \\
& =   \frac{b_{-1}}{\Gamma(s)(s-1)}  + \frac{1}{\Gamma(s)}  \int_{0}^{1} \left[ \textrm{Str}K_M(t)- \frac{b_{-1}}{t} \right] t^{s-1}dt\,\,.
\end{align}
The first term in the right-hand side of (\ref{spectralzetacontinuation3}) is analytic except for a simple pole at $s=1$ with residue $b_{-1} = \textrm{vol}(M)/(4\pi)$.
The second term, call it $G_1(s)$, is analytic for $\mathrm{Re}(s) > 0$. We continue this term as follows
\begin{align}\label{spectralzetacontinuation4}
\notag G_1(s)   & =  \frac{1}{\Gamma(s)} \int_{0}^{1} \left[\frac{\textrm{Str}K_M(t)}{t} - \frac{b_{-1}}{t^2}\right]  t^{s}dt \\
& =   \frac{b_{0}}{\Gamma(s)s}  + \frac{1}{\Gamma(s)}  \int_{0}^{1} \left[ \frac{\textrm{Str}K_M(t)}{t}- \frac{b_{-1}}{t^2} - \frac{b_{0}}{t}\right] t^{s}dt\,\,.
\end{align}
The first term in the right-hand side of (\ref{spectralzetacontinuation4}) is entire,
while the second term, call it $G_2(s)$, is analytic for $\mathrm{Re}(s) > -1$. By the $n$-th iterate ($n=0,1,2,\ldots$), the function $G(s)$ satisfies the formula
\begin{align*}
G(s)   = \sum_{k=-1}^{n-1} \frac{b_{k}}{\Gamma(s)(s+k)}  +  
\frac{1}{\Gamma(s)}  \int_{0}^{1} \left[ \frac{\textrm{Str}K_M(t)}{t^n}- \sum_{k=-1}^{n-1} \frac{b_{k}}{t^{n-k}}  \right] t^{s+n-1}dt\,\,
\end{align*}
with the right-hand side being analytic for $\textrm{Re}(s)> -n$.  In this fashion, the spectral zeta can be continued to all $s\in \mathbb{C}\,\,.$
\end{proof}

\hskip 0.2in If the surface $M$ is not compact, one defines the spectral zeta by the Mellin transform of the standard trace as in the formula (\ref{spectralzetadefinition}) above. Similar arguments may be employed to show the analytic continuation of the spectral zeta associated to a non-compact surface.

\hskip 0.2in For $\alpha\in(0,1/4)$ we define the $\alpha$-truncated standard trace by
\begin{align*}
\textrm{Str}K_M^{(\alpha)}(t) = \textrm{Str}K_M(t) - \sum_{\lambda_{n} < \alpha } e^{-\lambda_{n} t}\,\,.
\end{align*}
By considering the Mellin transform of the
standard trace we can express the truncated spectral zeta function
as
\begin{align*}
\zeta_{M}^{(\alpha)}(s) = \frac{1}{\Gamma(s)}\int_{0}^{\infty}
\textrm{Str}K_M^{(\alpha)}(t)t^s\frac{dt}{t}.
\end{align*}
With these in mind, we have the following result concerning the behavior of the truncated spectral zeta function through elliptic degeneration.

\begin{thm}\label{spectralzetaconvergence}
Let $M_{q}$ be an elliptically degenerating sequence of compact
or non-compact hyperbolic Riemann surfaces of finite volume with
limiting surface $M_{\infty}$. Let $\alpha < 1/4$ be any number that
is not an eigenvalue of $M_{\infty}$. Then for any $s \in
\mathbb{C}$, we have
\begin{align*}
\lim_{q\to \infty} \Bigg[\zeta_{M_{q}}^{(\alpha)}(s)-
\frac{1}{\Gamma(s)}\int_{0}^{\infty}
\mathrm{Dtr}K_{M_{q}}(t)t^s\frac{dt}{t}\Bigg] =
\zeta_{M_{\infty}}^{(\alpha)}(s).
\end{align*}
Furthemore, the convergence is uniform in half planes of the form
$\mathrm{Re}(s)>C$.
\end{thm}
\begin{proof}
We have to show that
\begin{align*}
\lim_{q\to \infty} \frac{1}{\Gamma(s)} \int_{0}^{\infty} \left[
\textrm{Str}K_{M_{q}}^{(\alpha)}(t)- \textrm{Dtr}K_{M_{q}}(t)
\right]t^s\frac{dt}{t}=  \frac{1}{\Gamma(s)}\int_{0}^{\infty}\textrm{Str}K_{M_{\infty}}^{(\alpha)}(t)t^s\frac{dt}{t}\,\,.
\end{align*}

Recalling Definitions (\ref{standardheattrace}) and (\ref{alphatruncatedtrace}), the bracket in the left hand side above may be broken down as follows
\begin{align}\label{spectralzetaconv1}
\notag \textrm{Str}K_{M_{q}}^{(\alpha)}(t)- \textrm{Dtr}K_{M_{q}}(t)
  & = \text{vol}(M_q)K_{\mathbb H}(t,0) \\
 + & \left[ \textrm{Htr}K_{M_{q}}(t)  + \textrm{Etr}K_{M_{q}}(t)  -  \sum_{\lambda_{q,n} < \alpha } e^{-\lambda_{q,n} t} -\textrm{Dtr}K_{M_{q}}(t)
\right] \,\,.
\end{align}

For the volume containing term in the right hand side (\ref{spectralzetaconv1}), we split the integral as
\begin{align*}
\frac{1}{\Gamma(s)} \int_{0}^{\infty} \text{vol}(M_q)K_{\mathbb H}(t,0) t^s\frac{dt}{t} 
= \frac{\text{vol}(M_q)}{\Gamma(s)} \left[\int_{0}^{1} K_{\mathbb H}(t,0) t^s\frac{dt}{t}  + 
\int_{1}^{\infty} K_{\mathbb H}(t,0) t^s\frac{dt}{t} \right]
\end{align*}
and make the following remarks. The volume is bounded by a universal constant depending solely on the genus and the total number $\kappa$ of cusps and conical ends of the family, namely $\text{vol}(M_q) \le 2\pi(2g-2+\kappa)$. 
By (\ref{identitytraceasymp}), the kernel function $K_{\mathbb H}(t,0)$ decays exponentially as $t$ goes to infinity, so that the integral over $[1,\infty)$ is entire as a function of $s$. Using the same arguments as in the course of the proof of Proposition \ref{spectralzetaprop}, the integral over $[0,1]$ is analytic
for $\mathrm{Re}(s)>1$ and may be continued to all $s \in \mathbb{C}$. 

\hskip 0.2in The integral consisting of the rest of the terms in (\ref{spectralzetaconv1}), namely 
\begin{align}\label{spectralzetaconv2}
\frac{1}{\Gamma(s)} \int_{0}^{\infty} \left[
\textrm{Htr}K_{M_{q}}(t)  + \textrm{Etr}K_{M_{q}}(t)  -  \sum_{\lambda_{q,n} < \alpha } e^{-\lambda_{q,n} t} -\textrm{Dtr}K_{M_{q}}(t)
\right]t^s\frac{dt}{t} 
\end{align}
can be split over $[0,1]$ and $[1,\infty)$. 
From Theorem \ref{uniformlongtimeasymptotics}, the bracket in (\ref{spectralzetaconv2}) has exponential decay; so then the portion over $[0,1]$ is analytic for $\mathrm{Re(s)>0}$ and may be continued
to the whole complex plane, while the part of the integral over $[1,\infty)$ is entire as function of $s$. 
By the dominated convergence theorem, we can interchange the limit and
the integral. The proof then follows by the convergence Theorem \ref{pointwiseanduniformconvergence}.
\end{proof}

\subsection{Hurwitz zeta function}

\hskip 0.2in 
As in the case of the spectral zeta function, we start in the compact setting where the Hurwitz zeta function is represented via the 
Dirichlet series
\begin{align*}
\zeta_M(s,z) = \sum_{\lambda_n>0} (z+\lambda_n)^{-s}\,\,,
\end{align*}
for $z,s \in \mathbb C$ with $\mathrm{Re}(z) > 0$ and $\mathrm{Re}(s) > 1$.

\hskip 0.2in In the case when $M$ is compact and connected, the Hurwitz zeta function is essentially the Laplace-Mellin transform of the standard trace of the heat kernel
\begin{align}\label{Hurwitzzetaintegraltransform}
\zeta_M(s,z) =
\frac{1}{\Gamma(s)}\int_{0}^{\infty}[\textrm{Str}K_M(t)-1]e^{-zt}t^s\frac{dt}{t}\,\,.
\end{align}
The above integral transform allows to extend the definition of the Hurwitz zeta function to the non-compact setting.

\hskip 0.2in From Section1 of \cite{JoLa 93} (see also \cite{Sa 87}) we obtain the following result.
\begin{prop}
{For each $z \in \mathbb{C}$, the Hurwitz zeta function extends to a
meromorphic function to all $s\in \mathbb{C}$. }
\end{prop}
\begin{proof}
Assuming first that $z > 0$ we expand the right-hand side of (\ref{Hurwitzzetaintegraltransform}) as follows
\begin{align}\label{Hurwitzzetaexpansion}
\notag \zeta_M(s,z) =& \frac{1}{\Gamma(s)} \int_{1}^{\infty}\left[\textrm{Str}K_M(t)-1\right]e^{-zt}t^{s-1}dt\\
\notag &+\frac{1}{\Gamma(s)} \int_{0}^{1} \textrm{Str} K_M(t) e^{-zt}t^{s-1}dt\\
&-\frac{1}{\Gamma(s)z^{s}} \left[ \Gamma(s) - \int_{z}^{\infty}e^{-t}t^{s-1}dt\right]\,\,.
\end{align}
By (\ref{standardtraceasympinfty}), the first term in the right hand side of (\ref{Hurwitzzetaexpansion}) above is entire as a function of $s$. 
For the second term, which is initially defined for $\mathrm{Re}(s) > 1$, we can mimic the ideas starting with (\ref{spectralzetacontinuation2}) in Proposition \ref{spectralzetaprop} to provide its analytic continuation. The third term is entire as a function of $s$. Consequently, these arguments allow to extend the Hurwitz zeta function to $\mathrm{Re}(z)>0$. 

\hskip 0.2in Next, we extend the Hurwitz zeta to $\mathrm{Re}(z) > -\lambda_1$. In this direction, we have 
\begin{align}\label{Hurwitzzetaexpansion2}
\notag \zeta_M(s,z) =& \frac{1}{\Gamma(s)} \int_{0}^{\infty}e^{\lambda_1 t}\left[\textrm{Str}K_M(t)-1\right]e^{-(z+\lambda_1)t}t^{s-1}dt\\
 =& \sum_{0 <\lambda_n\le \lambda_1} (z+\lambda_n)^{-s} + 
 \frac{1}{\Gamma(s)}\int_{0}^{\infty}\left[ \sum_{\lambda_n > \lambda_1} e^{-(\lambda_n-\lambda_1)t} \right]e^{-(z+\lambda_1)t}t^{s-1}dt\,\,.
\end{align}
The first sum in the right-hand side of (\ref{Hurwitzzetaexpansion2}) has finitely many terms (according to the multiplicity of $\lambda_1$). 
For the second term, the sum in the bracket has the same asymptotic behavior as $\textrm{Str}K_M(t)-1$. Consequently, the second term is now defined 
for $\mathrm{Re}(z)> -\lambda_1$ and can be continued to all $s \in \mathbb{C}$. The process then can be repeated to extend to $\mathrm{Re}(z)> -\lambda_k$,
with $\lambda_k$ being the first eigenvalue surpassing $\lambda_1$.
\end{proof}

\hskip 0.2in
We end this section by presenting the behavior of the Hurwitz zeta through degeneration. 
For $\alpha\in(0,1/4)$ we define the $\alpha$-truncated Hurwitz zeta function as
\begin{align*}
\zeta_{M}^{(\alpha)}(s,z) = \sum_{\lambda_n>\alpha} (z+\lambda_n)^{-s}  = \frac{1}{\Gamma(s)}\int_{0}^{\infty}
\textrm{Str}K_M^{(\alpha)}(t)e^{-zt}t^s\frac{dt}{t}\,\,.
\end{align*}
With these in mind, we have the following result concerning the behavior of the truncated spectral zeta function through elliptic degeneration.

\begin{thm}\label{Hurwitzzetaconvergence}
Let $M_{q}$ be an elliptically degenerating sequence of compact
or non-compact hyperbolic Riemann surfaces of finite volume with
limiting surface $M_{\infty}$. Let $\alpha < 1/4$ be any number that
is not an eigenvalue of $M_{\infty}$. Then for any $s,z \in
\mathbb{C}$, we have
\begin{align*}
\lim_{q\to \infty} \Bigg[\zeta_{M_{q}}^{(\alpha)}(s,z)-
\frac{1}{\Gamma(s)}\int_{0}^{\infty}
\mathrm{Dtr}K_{M_{q}}(t)e^{-zt}t^s\frac{dt}{t}\Bigg] =
\zeta_{M_{\infty}}^{(\alpha)}(s,z).
\end{align*}
Furthemore, the convergence is uniform in half planes of the form
$\mathrm{Re}(s)>C$ and fixed $z$.
\end{thm}

\begin{proof}
The result follows using similar arguments as in Theorem \ref{spectralzetaconvergence}.
\end{proof}

\section{Selberg zeta and determinant of the Laplacian}
\label{Selberg zeta and determinant of the Laplacian}

\hskip 0.2in
In this section, we investigate the behavior of the Selberg zeta function and the determinant of the Laplacian. After we recall definitions and some analytic properties of these functions, we describe their asymptotics through elliptic degeneration. It is worth mentioning that the spectral zeta, Selberg zeta, and the determinant of the Laplacian, are very much connected. Namely, the determinant of the Laplacian specialized to $s(s-1)$ is essentially the completed Selberg zeta function, with additional factors coming from the volume and the conical points (\cite{Sa 87}, \cite{Vo 87}, \cite{Ko 91}), while the spectral zeta function regularizes the determinant product. This comes with no surprise since the aforementioned functions appear in either the spectral side or the geometric side of the trace formula.

\subsection{Selberg zeta function}
\label{Selberg zeta function}

\hskip 0.2in
The Selberg zeta function is defined by the product
\begin{align*}
Z_M(s) = \prod_{\gamma \in
H(\Gamma)}\prod_{n=0}^{\infty}\Big(1-e^{-(s+n)\ell_{\gamma}}\Big).
\end{align*}
Following an elementary argument (see for example Lemma 4 in
\cite{JoLu 95}), one can estimate the number of closed geodesics of
bounded length. It then follows that the Euler product which defines
the Selberg zeta function converges for $\mathrm{Re}(s)>1$.

\hskip 0.2in
Following \cite{McK 72}, the integral representation is derived by carefully manipulating the logarithmic derivative of the Selberg zeta, namely
\begin{align*}
\frac{Z_{M}'(s)}{Z_{M}(s)}&= \sum_{\gamma
\in H(\Gamma)} \sum_{n=0}^{\infty}\frac{\ell_{\gamma}  e^{-(s+n)\ell_{\gamma}}}{1-e^{-(s+n)\ell_{\gamma}}} \\
&= \sum_{\gamma
\in H(\Gamma)} \sum_{n=0}^{\infty} \sum_{k=1}^{\infty} \ell_{\gamma} e^{-(s+n)\ell_{\gamma} k }\\
&= \sum_{\gamma
\in H(\Gamma)} \sum_{k=1}^{\infty} \frac{\ell_{\gamma} e^{-s k \ell_{\gamma}}}{1-e^{- k \ell_{\gamma}}} \\
&=\sum_{\gamma \in H(\Gamma)}\sum_{n=1}^{\infty}
\frac{\ell_{\gamma}}{2
\sinh(n\ell_{\gamma}/2)}e^{-(s-1/2)n\ell_{\gamma}}\,\,.
\end{align*}

Recalling the definition of the $K$-Bessel function
\begin{align*}
K_s(a,b) = \int_{0}^{\infty}e^{-(a^2t+b^2/t)} t^s \frac{dt}{t},
\end{align*}
as well as the fact that $K_{1/2}(b,a)=K_{-1/2}(a,b) =
(\sqrt{\pi}/b) e^{-2ab}$, allows us to further write

\begin{align*}
\frac{Z_{M}'(s)}{Z_{M}(s)}&=(2s-1) \sum_{\gamma
\in H(\Gamma)} \sum_{n=1}^{\infty}\frac{\ell_{\gamma}}{\sqrt{16\pi}
\sinh(n\ell_{\gamma}/2)}K_{1/2}(s-1/2,n\ell_{\gamma}/2)\\
&=(2s-1) \int_{0}^{\infty} \left[ \sum_{\gamma \in H(\Gamma)}\sum_{n=1}^{\infty}
\frac{\ell_{\gamma} e^{-(t/4+(n\ell_{\gamma})^2/(4t))}}{\sqrt{16\pi} \sinh(n\ell_{\gamma}/2)} \right] e^{-s(s-1)t} dt.
\end{align*}

Using the expression for the
hyperbolic heat trace (\ref{hyperbolictrace}),
the logarithmic derivative of the Selberg zeta function can be
expressed via the integral
\begin{align*}
\frac{Z_{M}'(s)}{Z_{M}(s)}=(2s-1)\int_{0}^{\infty}\textrm{Htr}K_{M}(t)e^{-s(s-1)t}dt\,\,.
\end{align*}

\hskip 0.2in 
For $\alpha < 1/4$, we define the $\alpha$-truncated logarithmic
derivative of the Selberg zeta function, using the above integral representation minus the contribution to the trace
of the small eigenvalues. Consequently,  we have
\begin{align*}
\frac{Z_{M}^{(\alpha)'}(s)}{Z_{M}^{(\alpha)}(s)}=
(2s-1)\int_{0}^{\infty}\textrm{Htr}K_{M}^{(\alpha)}(t)e^{-s(s-1)t}dt=
\frac{Z_{M}'(s)}{Z_{M}(s)}- \sum_{\lambda_{M,n}<\alpha}
\frac{2s-1}{s(s-1)+\lambda_{M,n}}\,\,,
\end{align*}
for $\mathrm{Re}(s)>1$ or $\mathrm{Re}(s^2-s) > -1/4$.  

\begin{thm}\label{logderivativeselbergzetaconvergence}
Let $M_{q}$ be an elliptically degenerating sequence of compact
or non-compact hyperbolic Riemann surfaces of finite volume with
limiting surface $M_{\infty}$. Let $\alpha < 1/4$ be any number that
is not an eigenvalue of $M_{\infty}$. Then, for any $s$ with
$\mathrm{Re}(s)>1$ or $\mathrm{Re}(s^2-s) > -1/4$, we have
\begin{align*}
\lim_{q\to \infty}
\frac{Z_{M_{q}}^{(\alpha)'}(s)}{Z_{M_{q}}^{(\alpha)}(s)}
=\frac{Z_{M_{\infty}}^{(\alpha)'}(s)}{Z_{M_{\infty}}^{(\alpha)}(s)}.
\end{align*}
\end{thm}
\begin{proof}
The proof follows from the integral representation of the
logarithmic derivative of the Selberg zeta function to which we
apply similar arguments as in Theorem \ref{spectralzetaconvergence}.
\end{proof}
As a direct corollary to Theorem
\ref{logderivativeselbergzetaconvergence} we obtain the following
result.

\begin{cor}
Let $M_{q}$ be an elliptically degenerating sequence of compact
or non-compact hyperbolic Riemann surfaces of finite volume with
limiting surface $M_{\infty}$.
\begin{enumerate}[(a)]
\item For any $s$ with $\mathrm{Re}(s)>1$ or $\mathrm{Re}(s^2-s)>-1/4$, we have
\begin{align*}
\lim_{q\to \infty} Z_{M_{q}}(s) = Z_{M_{\infty}}(s).
\end{align*}
\item At $s=1$, we have
\begin{align*}
\lim_{q\to \infty} Z_{M_{q}}'(1) = Z_{M_{\infty}}'(1).
\end{align*}
\end{enumerate}
\end{cor}

\subsection{Determinant of the Laplacian} 

\hskip 0.2in 
For a compact surface $M$, the determinant of Laplacian $\Delta_M$ is defined as the infinite product
\begin{align}\label{determinant}
\mathrm{det\,} \Delta_M = \prod_{\lambda_n > 0 } \lambda_n \,\,,
\end{align}
(see for instance \cite{Sa 87}, \cite{Vo 87}, \cite{JoLa 93}, \cite{JoLa 93b}, \cite{JoLa 94}, \cite{Ts 97} ). 
To give meaning to such divergent product, we observe that if the above product converged, than the logarithm of the determinant  
could be written as    
\begin{align*}
- \log \mathrm{det\,} \Delta_M = - \sum_{\lambda_n > 0 } \log \lambda_n  = 
\frac{d}{ds} \sum_{\lambda_n > 0 } \lambda_n^{-s} \Big|_{s=0} = \zeta_{M}'(0)\,\,.
\end{align*}
Recalling from Proposition \ref{spectralzetaprop} that the spectral zeta $\zeta_M(s)$ is analytic at $s=0$, the above formal manipulation suggests that the divergent product in (\ref{determinant}) be regularized as 
\begin{align}\label{determinantvalue}
\mathrm{det\,} \Delta_M = \exp( -\zeta_{M}'(0) )\,\,.
\end{align}

\hskip 0.2in 
For $0<\alpha <1/4$, we can express the derivative of $\alpha$-truncated spectral zeta function as follows
\begin{align*}
\frac{d}{ds}
\zeta_{M}^{(\alpha)}(s) = -\frac{\Gamma'(s)}{\Gamma(s)^2} \int_{0}^{\infty}
\mathrm{Str}K_{M}^{(\alpha)}(t) t^s\frac{dt}{t}
+ \frac{1}{\Gamma(s)} \frac{d}{ds} \left(\int_{0}^{\infty}
\mathrm{Str}K_{M}^{(\alpha)}(t) t^s\frac{dt}{t}\right)\,\,.
\end{align*}
At $s=0$ the gamma function has a simple pole, so that $1/\Gamma(s) = 0$ and consequently the second term above has no contribution to the logarithmic determinant. Directly from the Weierstrass product definition of the gamma function, it follows that
\begin{align*}
\frac{\Gamma'}{\Gamma^2}(0) = \lim_{s\to 0} \frac{\Gamma'/\Gamma(s)}{\Gamma(s)} = \lim_{s\to 0} \frac{-\gamma-1/s}{1/s} = -1\,\,,
\end{align*}
where $\gamma$ denotes the Euler-Mascheroni constant. Consequently, the logarithmic determinant can be rewritten as
\begin{align}\label{logdetintrepresentation}
\log \mathrm{det}^{(\alpha)} \Delta_{M} = - \int_{0}^{\infty}
\textrm{Str}K_{M}^{(\alpha)}(t) \frac{dt}{t}\,\,.
\end{align}

\hskip 0.2in 
The integral representation (\ref{logdetintrepresentation})  above together with Theorem  \ref{spectralzetaconvergence} concerning the behavior of the spectral zeta through elliptic degeneration, yield the following result concerning the behavior of the regularized determinant.
\begin{cor}\label{logdeterminantlaplacianconvergence}
Let $M_{q}$ be an elliptically degenerating sequence of compact
or non-compact hyperbolic Riemann surfaces of finite volume with
limiting surface $M_{\infty}$. Let $\alpha < 1/4$ be any number that
is not an eigenvalue of $M_{\infty}$. Then
\begin{align*}
\lim_{q\to \infty} \Bigg[ \log \mathrm{det}^{(\alpha)}
\Delta_{M_{q}} + \int_{0}^{\infty}
\mathrm{Dtr} K_{M_{q}}(t) \frac{dt}{t}\Bigg] = \log
\mathrm{det}^{(\alpha)} \Delta_{M_{\infty}}.
\end{align*}
\end{cor}

\section{Integral kernels}
\label{Integral kernels}

\hskip 0.2in As in the articles \cite{HJL 97}, \cite{JoLu 97a}, and
\cite{JoLu 97b}, one can prove the asymptotic behavior of numerous
other spectral quantities having once established the heat kernel
convergence (see Theorem \ref{convergenceheatkernels}), and the regularized convergence
theorem of heat traces (see Theorem \ref{pointwiseanduniformconvergence}). For completeness, we list here
some of the questions that now can be answered and, for the sake of
brevity, we outline the method of proof. In the cases when Theorem
\ref{convergenceheatkernels} is used, one obtains a result which amounts to continuity
through degeneration.


\textbf{The resolvent kernel} \\
The resolvent kernel $g_M(w,x,y)$ is the integral kernel which
inverts the operator $\Delta-w$ on the orthogonal complement of the
null space of $\Delta -w$. In the case $w=0$, the resolvent kernel
becomes the classical Green's function. For $\mathrm{Re}(w)>0$ and
$x\ne y$, the resolvent kernel is defined by
\begin{equation*}
g_M(w,x,y) = -\int_{0}^{\infty}K(t,x,y) e^{-wt}dt.
\end{equation*}

\hskip 0.2in If the surface is compact, we can use the spectral expansion of the heat kernel as in equation (\ref{heatkernelexpansioncompact}) to write
\begin{align*}
g_M(w,x,y) = -\sum_{n=0}^{\infty} \Bigg(\frac{1}{w+\lambda_{M,n}}\Bigg)\phi_{M,n}(x)\phi_{M,n}(y)
\end{align*}
for $\mathrm{Re}(w)>0$ and $x\ne y$. From the above, it easily follows that the resolvent kernel has a meromorphic continuation to the entire plane with poles located at the negative eigenvalues of the Laplacian.
If the surface is not compact, there is a similar spectral expansion for the resolvent kernel (coming from (\ref{heatkernelexpansionnoncompact}) together with the above integral representation).

\hskip 0.2in Let $0<\alpha<1/4$. Then the $\alpha$-truncated resolvent kernel $g_M^{(\alpha)}(w,x,y)$ is given by
\begin{align*}
g_M^{(\alpha)}(w,x,y) = g_M(w,x,y) + \sum_{\lambda_{M,n} < \alpha} \Bigg(\frac{1}{w+\lambda_{M,n}}\Bigg)\phi_{M,n}(x)\phi_{M,n}(y).
\end{align*}
It then follows that the truncated resolvent kernel inverts $\Delta+w$ on the orthogonal complement of the space spanned by the eigenfunctions that correspond to the eigenvalues of $\Delta$ which are less than $\alpha$.

\hskip 0.2in With the above remarks in mind, we have the following result.
\begin{thm}
Let $M_{q}$ be an elliptically degenerating sequence of compact
or non-compact hyperbolic Riemann surfaces of finite volume with
limiting surface $M_{\infty}$. Let $0< \alpha < 1/4$.
\begin{enumerate}[(a)]
\item For all fixed $w$ with $\mathrm{Re}(w)>0$, we have
\begin{align*}
\lim_{q\to \infty} g_{M_{q}}(w,x,y) = g_{M_{\infty}}(w,x,y).
\end{align*}
The convergence is uniform for $x\ne y$ bounded away from the developing cusps and in half-planes $\mathrm{Re}(w)>0$.
\item For all fixed $w$ with $\mathrm{Re}(w)>-\alpha$, we have
\begin{align*}
\lim_{q\to \infty} g_{M_{q}}^{(\alpha)}(w,x,y) = g_{M_{\infty}}^{(\alpha)}(w,x,y).
\end{align*}
The convergence is uniform for $x\ne y$ bounded away from the developing cusps and in half-planes $\mathrm{Re}(w)>-\alpha$.
\end{enumerate}
\end{thm}
\begin{proof}
Part (a) follows from the convergence of the heat kernel as in Proposition \ref{convergenceheatkernels} together with the dominated convergence theorem. Part (b) is similar to part (a) with the addition of the convergence of the small eigenvalues and eigenfunctions from Section \ref{Convergence of spectral counting functions and small eigenvalues}.
\end{proof}

\textbf{The Poisson kernel} \\
A Poisson kernel on the surface $M$ is a smooth function $P_M(w,x,y)$ defined on $\mathbb{R}^{+}\times M\times M$, satisfying the following conditions: Suppose that $f$ is a bounded and continuous function on $M$ and define
\begin{align*}
u(w,x) = \int_{M} P_M(w,x,y) f(y)d\mu(y).
\end{align*}
Then the Poisson kernel satisfies the differential equation
\begin{align*}
(\Delta_x -\partial_w^2)u(w,x)=0
\end{align*}
and the Dirac condition
\begin{align*}
f(x)=\lim_{w\to 0^+} \int_{M} P_M(w,x,y) f(y)d\mu(y)
\end{align*}
uniformly on compact sets. For a more detailed discussion on the Poisson kernel we refer the reader to \cite{JoLa 03}.

\hskip 0.2in The Poisson kernel is given through the
G-transform
\begin{equation*}
P_M(w,x,y) = \frac{w}{\sqrt{4\pi}}\int_{0}^{\infty}K_M(t,x,y)
e^{-w^2/4t}t^{-3/2}dt.
\end{equation*}
We conclude continuity of the Poisson kernel through degeneration. By
arguing as in the case of the resolvent kernel mentioned above, the
region of convergence extends to all $w \in \mathbb{C}$.\\

\textbf{The wave kernel}\\
From the Poisson kernel we can define the wave kernel by rotation of the $w$-variable, namely
\begin{equation*}
W_M(w,x,y)=P_M(-iw,x,y).
\end{equation*}
The wave kernel $W_M(w,x,y)$ is a fundamental solution to the wave equation
\begin{equation*}
\Delta_x+\partial_w^2 =0.
\end{equation*}
As with the Poisson kernel, we obtain continuity of the wave kernel through degeneration.\\

\newpage

\addcontentsline{toc}{section}{Bibliography}
\renewcommand\refname{Bibliography}

\vspace{5mm}\noindent
Daniel Garbin \\
Department of Mathematics and Computer Science \\
Queensborough Community College\\
222-05 56th Avenue \\
Bayside, NY 11364 \\
U.S.A. \\
e-mail: dgarbin@qcc.cuny.edu

\vspace{5mm} \noindent
Jay Jorgenson \\
Department of Mathematics \\
The City College of New York \\
Convent Avenue at 138th Street \\
New York, NY 10031\\
U.S.A. \\
e-mail: jjorgenson@mindspring.com

\end{document}